\documentclass[12pt]{amsart}
\usepackage{amssymb,latexsym,amsmath,tikz-cd}
\usepackage{enumitem}
\usepackage[mathscr]{eucal}
\usepackage{bbm,xcolor,comment}
\usepackage[vcentermath]{youngtab}
\usepackage{mathdots}
\usepackage{chngcntr}

\numberwithin{equation}{section}
\numberwithin{table}{section} 

\newcommand{\eref}[1]{\emph{\ref{#1}}}

\newcommand{\hsp}[1]{{\hbox{\hspace{#1}}}}

\newcounter{letcnt1} 

\newcounter{letcnt2} 
\counterwithin{letcnt2}{letcnt1}

\newcounter{talkcnt} 

\def\a{\alpha}  
\def\b{\beta}  

\def\m{\mu}

\def\w{\omega}

\def\ttA{\mathtt{A}}

\def\tAd{\mathrm{Ad}} \def\tad{\mathrm{ad}}
 
\def\tAut{\mathrm{Aut}}

\def\bC{\mathbb C}

\def\cD{\mathcal D}

 \def\tdim{\mathrm{dim}}
 
 \def\ttE{\mathtt{E}}

\def\tEnd{\mathrm{End}} 

 \def\fe{\mathfrak{e}}

\def\ff{\mathfrak{f}}

\def\tFlag{\mathrm{Flag}}

  \def\bfG{\mathbf{G}} 

\def\cG{\mathcal G} 

\def\tGr{\mathrm{Gr}}
\def\fg{{\mathfrak{g}}}

 \def\sH{\mathscr{H}}

\def\fh{\mathfrak{h}} 
\def\bh{\mathbf{h}}
\def\bi{\mathbf{i}} 
\def\tti{\mathtt{i}}

\def\tId{\mathrm{Id}}

\def\fk{\mathfrak{k}}

\def\cM{\mathcal M}

\def\tO{\mathrm{O}} 

\def\fo{\mathfrak{o}} 
\def\bP{\mathbb P} \def\cP{\mathcal P}

\def\fp{\mathfrak{p}}

\def\bQ{\mathbb Q} \def\cQ{\mathcal Q}

\def\bR{\mathbb R}

\def\tss{\mathrm{ss}}
 \def\tSO{\mathrm{SO}}
\def\tSp{\mathrm{Sp}}

\def\tStab{\mathrm{Stab}} \def\tSym{\mathrm{Sym}}

 \def\tspan{\mathrm{span}}
\def\fsl{\mathfrak{sl}} \def\fso{\mathfrak{so}} 
\def\fsp{\mathfrak{sp}} \def\fsu{\mathfrak{su}}
  
 \def\ttT{\mathtt{T}}

  \def\tU{\mathrm{U}}

\def\fu{\mathfrak{u}}

   \def\bZ{\mathbb Z}
 
\def\fz{\mathfrak{z}} 
 
\def\half{\tfrac{1}{2}}

\def\fourth{\tfrac{1}{4}}

\def\tand{\quad\hbox{and}\quad}

\def\smallb{{\hbox{\small{$\bullet$}}}}

\def\inj{\hookrightarrow}

\def\op{\oplus}
\def\ot{\otimes}

\def\tw{\hbox{\small $\bigwedge$}}

\newenvironment{a_list}
  {\begin{enumerate}[label=(\alph*),itemsep=3pt,leftmargin=25pt,listparindent=20pt]}
  {\end{enumerate}}

\newenvironment{A_list}
  {\begin{enumerate}[label=(\Alph*),itemsep=3pt,leftmargin=25pt,listparindent=20pt]}
  {\end{enumerate}}

\newenvironment{numlist}
  {
  \begin{enumerate}[itemsep=3pt,leftmargin=25pt,listparindent=20pt]
  }
  {\end{enumerate}}
\newenvironment{i_list}
  {\begin{enumerate}[label=(\roman*),itemsep=3pt,leftmargin=25pt,listparindent=20pt]}
  {\end{enumerate}}
\newenvironment{i_list_emph}
  {\begin{enumerate}[label=\emph{(\roman*)},itemsep=3pt,leftmargin=25pt,listparindent=20pt]}
  {\end{enumerate}}

\newtheorem{proposition}[equation]{Proposition}
\newtheorem{theorem}[equation]{Theorem}

\newtheorem*{theorem*}{Theorem}

\theoremstyle{definition}

\newtheorem*{boldQ*}{Question}
\newtheorem*{boldP*}{Problem}

\theoremstyle{definition}

\theoremstyle{remark}
\newtheorem*{assume*}{Assume}
\newtheorem*{answer*}{Answer}

\newtheorem*{claim*}{Claim}

\newtheorem{definition}[equation]{Definition}
\newtheorem*{definition*}{Definition}
\newtheorem{example}[equation]{Example}
\newtheorem*{example*}{Example}

\newtheorem*{hint*}{Hint}
\newtheorem*{notation*}{Notation}
\newtheorem{remark}[equation]{Remark}
\newtheorem*{remark*}{Remark}
\newtheorem*{remarks*}{Remarks}
\newtheorem*{fact*}{Fact}

\newtheorem*{emphQ*}{Question}
\newtheorem*{emphA*}{Answer}

\begin{document}
\title{Hodge representations}
\author[Han]{Xiayimei Han}
\email{xiayimei.han@duke.edu}
\author[Robles]{Colleen Robles}
\email{robles@math.duke.edu}
\address{Mathematics Department, Duke University, Box 90320, Durham, NC  27708-0320} 
\thanks{Robles is partially supported by NSF grants DMS 1611939 and 1906352.}
\date{\today}
\begin{abstract}
Hodge representations were introduced by Green--Griffiths--Kerr to classify the  Hodge groups of polarized Hodge structures, and the corresponding Mumford--Tate subdomains of a period domain.  The purpose of this article is to provide an exposition of how, given a fixed period domain $\cD$, to enumerate the Hodge representations corresponding to Mumford--Tate subdomains $D \subset \cD$.  After reviewing the well-known classical cases that $\cD$ is Hermitian symmetric (weight $n=1$, and weight $n=2$ with $p_g = h^{2,0}=1$), we illustrate this in the case that $\cD$ is the period domain parameterizing polarized Hodge structures of (effective) weight two Hodge structures with first Hodge number $p_g = h^{2,0} = 2$.  We also classify the Hodge representations of Calabi--Yau type, and enumerate the horizontal representations of CY 3-fold type.  (The ``horizontal'' representations those with the property that corresponding domain $D \subset \cD$ satisfies the infinitesimal period relation, a.k.a.~Griffiths' transversality, and is therefore Hermitian.) 
\end{abstract}
\maketitle

\subsection*{Conflicts of Interest}  The authors declare none.
\subsection*{Data Availability Statement}  All data present or explicitly referenced.

\setcounter{tocdepth}{1}
\let\oldtocsection=\tocsection
\let\oldtocsubsection=\tocsubsection
\let\oldtocsubsubsection=\tocsubsubsection
\renewcommand{\tocsection}[2]{\hspace{0em}\oldtocsection{#1}{#2}}
\renewcommand{\tocsubsection}[2]{\hspace{3em}\oldtocsubsection{#1}{#2}}
\tableofcontents 

\section{Introduction}

\subsection{Hodge groups}

Fix a period domain $\cD = \cD_\bh = \cG_\bR/\cG_\bR^0$ parameterizing $Q$--polarized Hodge structures on a rational vector space $V$ with Hodge numbers $\bh = (h^{n,0},\ldots,h^{0,n})$.  Here 
\[
  \cG_\bR \ = \ \tAut(V_\bR,Q)
\]
is either an orthogonal group $\tO(a,2b)$ (if $n$ is even) or a symplectic group $\tSp(2r,\bR)$ (if $n$ is odd), and $\cG_\bR^0$ is the compact stabilizer of a fixed $\varphi \in \cD$.  To each Hodge structure $\varphi \in \cD$ is associated a ($\bQ$--algebraic) Hodge group $\bfG_\varphi \subset \tAut(V,Q)$, and a Mumford--Tate domain $D = D_\varphi = G_\varphi \cdot \varphi \subset \cD$, where $G_\varphi = \bfG_\varphi(\bR)$.  Briefly, the Hodge structure $\varphi \in \cD$ determines a homomorphism of $\bR$--algebraic groups $\varphi : S^1 \to \tAut(V_\bR,Q)$, and the Hodge group $\bfG_\varphi$ is the $\bQ$--algebraic closure of $\varphi(S^1)$.  The Hodge group $\bfG_\varphi$ may be equivalently defined as the stabilizer of the Hodge tensors of $\varphi$.

\subsection{Motivations} \label{S:mot}
The geometric considerations motivating a classification of the Hodge groups for a given period domain $\cD$ include the following.  For generic choice of $\varphi \in D$, the Hodge group $\bfG_\varphi$ is the full automorphism group $\tAut(V,Q)$.  So when containment $\bfG_\varphi \subsetneq \tAut(V,Q)$ is strict, the Hodge structure has nongeneric Hodge tensors.  (And, because $\bfG_{\varphi'} \subset \bfG_\varphi$ for all $\varphi' \in D_\varphi$, the Mumford--Tate domain $D_\varphi$ will parameterize Hodge structures with nongeneric Hodge tensors.)  An extreme example here is the case that $D_\varphi$ is a point $\{\varphi\}$; this is the case if and only if $\bfG_\varphi$ is a torus; equivalently, $\tEnd(V,\varphi)$ is a CM field.  When containment $\bfG_\varphi \subsetneq \tAut(V,Q)$ is strict and the Hodge structure is realized by the cohomology of an algebraic variety, the variety ``should'' admit nongeneric arithmetic properties (such as extra Hodge classes, or automorphisms, et cetera). In general, the Hodge group can have significant geometric consequences; for example, it plays a key role in Ribet's study \cite{MR701568} of the Hodge conjecture for principally polarized abelian varieties (expanding upon earlier work of Tanke'ev's \cite{MR631439, MR643899}).

Likewise much geometric motivation for the classification of the Mumford--Tate domains comes from the moduli of algebraic varieties.   In general, the period domain is not Hermitian.  Two significant exceptions are the period domains arising when considering moduli spaces of principally polarized abelian varieties and K3 surfaces.  The Hermitian  symmetric structure of $\cD$ in these two cases, along with global Torelli theorems, is the underlying structure that has made Hodge theory such a powerful tool in the study of these moduli spaces and their compactifications \cite{MR3495110}.  Even when the period domain $\cD$ is not Hermitian, it may contain Hermitian symmetric Mumford--Tate subdomains $D$.  (For example, every horizontal subdomain is Hermitian symmetric.)  Given a moduli space $\cM$ geometrically realizing $D$ as a Mumford--Tate domain (with a Torelli theorem), Hodge theory is again a significant tool in the study of $\cM$ and its compactifications, cf.~\cite{MR1910264, MR2789835, MR1416355, MR2807280, MR1780433, MR3886178, MR3914747, MR2510071, MR1265318}.  Reciprocally, given a Hermitian symmetric Mumford--Tate domain $D \subset \cD$ it is a very interesting problem to find geometric (or motivic) realizations of the domain; work in this direction includes \cite{MR3474815, MR3914747, MR3193750, MR3337677, MR3350113}.

\subsection{Objective and approach} \label{S:approach}

The principal goal of this paper is to present an expository discussion of the Green--Griffiths--Kerr \cite{MR2918237} prescription to identify the real algebraic groups $G_\varphi = \bfG_\varphi(\bR)$ that may arise.  More precisely, Green--Griffiths--Kerr identify the underlying real Lie algebra $\fg_\bR$.  This determines $G_\varphi$ to finite data, and suffices to identify the domains $D_\varphi$ as intrinsic $G^\tad_\varphi$--homogeneous spaces.  (See \cite{MR3513533} for the classification of general $\bfG_\varphi$.)

\begin{example}\label{eg:wt1i}
The case of weight one Hodge representations is classical \cite{MR546620, MR2192012}: The real form $\fg_\bR$ is one of $\fsp_{2r}\bR$, $\fu(a,b)$, $\fsu(a,a)$, $\fso(2,m)$ and $\fso^*(2r)$. See Example \ref{eg:wt1} for the corresponding Hodge representations.
\end{example}

\begin{remark}
We are aware of only a few cases in which the classification of the $\bfG_\varphi$ as $\bQ$--algebraic groups has been completely worked out.  These include Zarhin's classification \cite{MR697317} of the Hodge groups of K3 surfaces (see Example \ref{eg:K3} for the corresponding Hodge representations), and Green--Griffiths--Kerr classification \cite[\S7]{MR2918237} for period domains $\cD$ with Hodge numbers $\bh = (2,2)$ and $\bh = (1,1,1,1)$.
\end{remark}

Green--Griffiths--Kerr \cite{MR2918237} showed that the Hodge groups $
\bfG = \bfG_\varphi$ and Mumford--Tate domains $D = D_\varphi \subset \cD_\bh$ are in bijection with Hodge representations
\[
  S^1 \ \stackrel{\phi}{\longrightarrow} \ G_\bR \,,\quad
  \bfG \ \inj \ \tAut(V,Q) \,,
\]
with Hodge numbers $\bh_\phi \le \bh$.  (Given $\bh_1 = ( h_1^{n,0} , \ldots , h_1^{0,n})$ and $\bh_2 = ( h_2^{n,0} , \ldots , h_2^{0,n})$, we write $\bh_1 \le \bh_2$ if $h_1^{p,q} \le h_2^{p,q}$ for all $p,q$.)  Effectively one may say that \emph{the Hodge groups and Mumford--Tate domains are classified by the Hodge representations}: given a fixed $\cD = \cD_\bh$ (with specified Hodge numbers $\bh$), one identifies all possible Hodge domains $D \subset \cD$ by enumerating the Hodge representations with $\bh_\phi \le \bh$.  

\begin{remark} \label{R:obvious}
There are some obvious subdomains that can be identified without Hodge representations: (products of) period subdomains.  If $\cD_i$ is the period domain for Hodge numbers $\bh_i$ and $\bh_1 + \cdots + \bh_\ell \le \bh$, then $\cD_1\times\cdots\times \cD_\ell$ is a Mumford--Tate subdomain of $\cD$.
\end{remark}

Green--Griffiths--Kerr's characterization of the Hodge representations is formulated as Theorem \ref{T:strategy}, which asserts that the induced (real Lie algebra) Hodge representations 
\begin{equation}\label{E:iHrep}
  \bR \ \to \ \fg_\bR \ \to \ \tEnd(V_\bR,Q)
\end{equation}
are enumerated by tuples $(\fg^\tss_\bC,\ttE^\tss,\mu,c)$ consisting of:
\begin{i_list}
\item a semisimple complex Lie algebra $\fg^\tss_\bC = [\fg_\bC,\fg_\bC]$, 
\item an element $\ttE^\tss \in \fg^\tss_\bC$ with the property that $\tad\,\ttE^\tss$ acts on $\fg^\tss_\bC$ diagonalizably with integer eigenvalues, 
\item a highest weight $\mu$ of $\fg^\tss_\bC$, and 
\item a constant $c \in \bQ$ satisfying $\mu(\ttE^\tss) + c \in \half\bZ$.
\end{i_list}
The real form $\fg_\bR^\tss$ is the Lie algebra of the image $G^\tad_\varphi$ of $\tAd : G_\varphi \to \tAut(\fg_\bR)$.  We have $D_\varphi = G^\tad_\varphi \cdot \varphi$, and $\ttE^\tss$ is essentially equivalent to the isotropy group $\tStab_{G^\tad_\varphi}(\varphi)$.

\subsection{Examples and special cases}
\subsubsection{Horizontal Hodge domains}

As discussed above (\S\ref{S:mot}) the identification of the horizontal subdomains is of particular interest.  These are the domains that satisfy the infinitesimal period relation (IPR, a.k.a.~Griffiths' transversality).  It is well-known that horizontal subdomains are necessarily Hermitian, and as such their structure as intrinsic homogeneous complex manifolds is classical and well-understood.  These results are reviewed in \S\ref{S:horiz}.  The Hodge representations with horizontal $D_\phi$ are characterized in Proposition \ref{P:horiz}.

\subsubsection{Weight two Hodge representations}

In \S\ref{S:eg} we apply the prescription of Theorem \ref{T:strategy} to identify all Hodge representations and Mumford--Tate subdomains $D$ of the period domain $\cD$ parameterizing $Q$--polarized, (effective) weight $n=2$ Hodge structures with $p_g = h^{2,0} = 2$ (Theorems \ref{T:pg=1}, \ref{T:pg=2}, \ref{T:notH} and Theorem \ref{T:ss}).  This period domain is chosen as our primary example for two reasons.  First, it is in a certain sense the simplest example of a period domain that is not Hermitian symmetric.  (The infinitesimal period relation is a contact subbundle of $T \cD$.)  Second, it is the period domain arising when considering families of Horikawa surfaces \cite{MR501370, MR517773}, in which there has been much interest recently \cite{MR3383166, MR3636379, MR3914747}.

\subsubsection{Hodge representations of Calabi--Yau type} 

Hodge representations of CY-type (those with first Hodge number $h^{n,0} = 1$) are of considerable interest and have been studied by several authors, including \cite{MR3102477, MR3189469, MR1258484, MR2657440}.  Much of this work is over $\bR$, but Friedman and Laza \cite{MR3189469} have identified some rational forms $\bfG_\varphi(\bQ)$ admitting Hodge representations of CY 3-fold type.  In \S\ref{S:CY} we classify the (Lie algebra) Hodge representations of CY-type (Theorem \ref{T:CY}).  The CY-Hodge representations with $D$ Hermitian are well-known, and those with $\fg_\bR$ semisimple have been classified \cite[Proposition 6.1]{MR3217458}; so the content of Theorem \ref{T:CY} is to drop the hypothesis that $\fg_\bR$ be semisimple from the classification.  This result is used in \cite{HHan} to enumerate the set of all Hodge representations of CY 3-fold type.  Those with horizontal (and therefore Hermitian) domain $D \subset \cD$ are listed in Example \ref{eg:CY3}.

\section{Hodge representations}\label{S:Hrep}

What follows is a laconic review of the necessary background material on Hodge representations.  References for more detailed discussion include \cite{MR2918237}, \cite[\S\S2--3]{MR3217458} and \cite[\S\S2--3]{MR3505643}.

\subsection{Basics} \label{S:basics}
 
Let
\begin{equation}\label{E:Hrep}
  \phi : S^1 \to G_\bR \tand
  G_\bR \ \to \ \tAut(V_\bR,Q)
\end{equation}
be the data of a (real) Hodge representation \cite{MR2918237}.  Without loss of generality, we may suppose that the induced Lie algebra representation
\begin{equation} \label{E:HrepLA}
   \fg_\bR \ \inj \ \tEnd(V_\bR,Q)
\end{equation}
is faithful.  The associated Hodge decomposition
\[
  V_\bC \ = \ \bigoplus V^{p,q}_\phi
\]
is the $\phi$--eigenspace decomposition; that is,
\[
  V^{p,q}_\phi \ = \ \{ v \in V_\bC \ | \ \phi(z)(v) = z^{p-q} \, v \,,\ 
  \forall \ z \in S^1 \} \,.
\]
The associated grading element $\ttE_\phi\in\bi\fg_\bR$ (or infinitesimal Hodge structure) \cite{MR3217458} is defined by $\ttE_\phi(v) = \half(p-q) v$ for all $v \in V^{p,q}_\phi$; that is, $\ttE_\phi \in \tEnd(V_\bC)$ is defined so that $V^{p,q}_\phi$ is the $\ttE_\phi$--eigenspace with eigenvalue $\half(p-q)$; for this reason it is sometimes convenient to write 
\[
  V^{p,q} \ = \ V_{(p-q)/2} \,.
\]  

\begin{remark}
Together the grading element $E$ and Lie algebra representation \eqref{E:HrepLA} determine the group representation \eqref{E:Hrep} up to finite data.
\end{remark}  

\begin{definition}
We call the the pair $(\fg_\bR \inj \tAut(V_\bR,Q) \,,\, \ttE)$ the data of a \emph{real, Lie algebra Hodge representation} ($\bR$-LAHR).
\end{definition}

\begin{remark}
A key point here is that a Hodge representation \eqref{E:Hrep} determines a grading element $\ttE_\phi \in \bi \fg_\bR$.  Conversely a complex reductive Lie algebra $\fg_\bC$, a grading element $\ttE \in \fg_\bC$ determines both a real form $\fg_\bR$ (\S\ref{S:gerf}) and a Hodge representation (\S\ref{S:gehr}).
\end{remark}

Notice that $\phi$ is a level $n$ Hodge structure on $V_\bR$ if and only if the $\ttE_\phi$--eigenspace decomposition is  
\begin{equation} \label{E:V}
  V_\bC \ = \ V_{n/2} \,\op\, V_{n/2-1} \,\op\cdots\op\, 
  V_{1-n/2} \,\op\, V_{-n/2} \,.
\end{equation}

\begin{remark} \label{R:phinot0}
The Hodge structure is of level zero (equivalently, $V_\bC = V_\phi^{0,0}$) if and only if $\phi$ is trivial.  We assume this is not the case.
\end{remark}

\begin{remark}[Period domains] \label{R:pd}
The Hodge domain $D$ determined by \eqref{E:Hrep} is a period domain if and only if $G_\bR = \tAut(V_\bR,Q)$.
\end{remark}

\subsection{Induced Hodge representation} \label{S:induced}

There is an induced Hodge representation on the Lie algebra $\fg_\bR$.  Define 
\[
  \fg^{\ell,-\ell}_\phi \ := \ 
  \{ \xi \in \fg_\bC \ | \ \xi(V^{p,q}_\phi) \subset V^{p+\ell,q-\ell}_\phi \,,\
   \forall \ p,q\}\,.
\]
Then 
\begin{equation}\label{E:ind}
 \fg_\bC \ = \ \bigoplus_{\ell\in\bZ} \fg^{\ell,-\ell}_\phi
\end{equation}
is a weight zero Hodge structure on $\fg_\bR$ that is polarized by $-\kappa$, with $\kappa$ the Killing form.  

The Jacobi identity implies 
\[
  [\fg^{k,-k}_\phi \,,\, \fg^{\ell,-\ell}_\phi] \ \subset \ \fg^{k+\ell,-k-l}_\phi \,.
\]
The subalgebra
\[
  \fg_{\phi,\bC}^\mathrm{even} \ := \ \bigoplus_{\ell \in \bZ} \fg^{2\ell,-2\ell}_\phi
\]
is the complexification $\fk_{\phi} \ot_\bR \bC$ of the (unique) maximal compact subalgebra $\fk_\phi \subset \fg_\bR$ containing the Lie algebra $\fg_\bR^0 = \fg_\bR \cap \fg^{0,0}_\phi$ of the stabilizer/centralizer $G_\bR^0$ of $\phi$.

\subsection{Grading elements}\label{S:ge}

Hodge structures are closely related to grading elements.  This relationship is briefly reviewed here; see \cite[\S\S2--3]{MR3217458} and \cite[\S\S2--3]{MR3505643} for details.

\begin{remark} \label{R:ge}
Here grading elements are essentially linearizations of the circle $\phi : S^1 \inj G_\bR$ in the Hodge representation \eqref{E:Hrep}.  The essential observation of this section is that the data $(\fg_\bC,\ttE)$ determines the real form $\fg_\bR$, and the Hodge domain and compact dual $D \subset \check D$ (as intrinsic homogeneous spaces.  They are not represented as subdomains of a period domain $\cD$ until we select the second half $G_\bR \inj \tAut(V_\bR,Q)$ of the Hodge representation \eqref{E:Hrep}).
\end{remark}

\subsubsection{Definition}\label{S:gedfn}

Fix a complex reductive Lie algebra $\fg_\bC$.  A \emph{grading element} is any element $\ttE \in \fg_\bC$ with the property that $\tad(\ttE) \in \tEnd(\fg_\bC)$ acts diagonalizably on $\fg_\bC$ with integer eigenvalues; that is,
\begin{equation}\label{E:ghs}
  \fg_\bC \ = \ \bigoplus_{\ell \in\bZ} \fg^{\ell,-\ell} \,,\quad\hbox{with}\quad
  \fg^{\ell,-\ell} \ = \ \{ \xi \in \fg_\bC \ | \ [\ttE , \xi] = \ell\,\xi \} \,.
\end{equation}

\begin{remark}
The notation $\fg^{\ell,-\ell}$ is meant to be suggestive.  The grading element $\ttE$ determines a weight zero (real) Hodge decomposition that is polarized by $-\kappa$, with $\kappa$ the Killing form (\S\ref{S:gehr}).
\end{remark}

\begin{remark}
The data $(\fg_\bC,\ttE)$ determines a parabolic subgroup $P_\ttE \subset G_\bC$ with Lie algebra $\fp_\ttE = \op_{\ell\ge0}\,\fg^{\ell,-\ell}$.  The resulting generalized grassmannian $\check D = G_\bC/P_\ttE$ (or rational homogeneous variety) is the compact dual of the Hodge domain (as an intrinsic homogeneous space).
\end{remark}

\subsubsection{Grading elements versus real forms}\label{S:gerf}

Fix a complex reductive Lie algebra $\fg_\bC$.   Given $\fg_\bC$ and $\ttE$, there is a unique real form $\fg_\bR$ of $\fg_\bC$ such that \eqref{E:ghs} is a weight zero Hodge structure on $\fg_\bR$ that is polarized by $-\kappa$ \cite[\S3.1.2]{MR3505643}.  The real form $\fg_\bR$ is determined by the condition that $\op_\ell\,\fg^{2\ell,-2\ell}$ is the complexification $\fk_\bC$ of a maximal compact subalgebra $\fk \subset \fg_\bR$.  

See \S\ref{S:horiz} for a discussion of the examples that of the most interest here.

\subsubsection{Grading elements versus Hodge representations} \label{S:gehr}

Given the data of \S\ref{S:gerf}, the grading element $\ttE$ acts on any representation $G_\bR \to \tAut(V_\bR)$ by rational eigenvalues.  The $\ttE$--eigenspace decomposition $V_\bC = \op_{k\in\bQ}\,V_{k}$ is a Hodge decomposition (polarized by some $Q$), with $V^{p,q} = V_{(p-q)/2}$ as in \S\ref{S:basics}, if and only if those eigenvalues lie in $\half \bZ$ \cite{MR2918237}.  The corresponding Hodge representation is given by the circle $\phi : S^1 \to G_\bR$ defined by 
\[
  \phi(z) v \ := \ z^{p-q}\,v \,,\quad 
  z \in S^1 \,,\ v \in V^{p,q} = V_{(p-q)/2} \,.
\]
Note that $\ttE = \ttE_\phi$.

\subsubsection{Normalization of grading element} \label{S:norm-ge}

The Lie algebra $\fg$ of $G$ is reductive.  Let 
\begin{equation} \label{E:gvgss}
  \fg \ = \ \fz \op \fg^\tss
\end{equation}
denote the decomposition of $\fg$ into its center $\fz$ and semisimple factor $\fg^\tss = [\fg,\fg]$.  Let $\ttE = \ttE' + \ttE^\tss$ be the decomposition given by \eqref{E:gvgss}.  The Hodge domain $D$ is determined by $\fg^\tss$ and $\ttE^\tss$. 

Fix a Cartan subalgebra $\fh \subset \fg_\bC$ that contains $\ttE_\phi$, is contained in $\fg^{0,0}_\phi$ and that is defined over $\bR$.  Then 
\[
  \fh \ = \ \fz \,\op\,\fh^\tss \,,
\]
where $\fh^\tss = \fh \cap \fg^\tss_\bC$ is a Cartan subalgebra of $\fg^\tss_\bC$.  Choose simple roots $\{ \a_1,\ldots,\a_r\} \in (\fh^\tss)^*$ of $\fg^\tss_\bC$ so that $\a_j(\ttE^\tss) \ge 0$, for all $j$.  Without loss of generality, we may assume that $\a_j(\ttE) \in \{0,1\}$.\cite[\S3.3]{MR3217458}\footnote{There is a typo in \cite[Proposition 3.4]{MR3217458}: in general one may assert only that the group $F$ is $\bR$--algebraic (not $\bQ$--algebraic).}  This is equivalent to the condition that the infinitesimal period relation $T^hD \subset TD$ is bracket--generating; equivalently, $\fg^{1,-1}$ generates $\fg^{+,-} = \op_{\ell>0}\,\fg^{\ell,-\ell}$ as a Lie algebra.

\subsection{Reduction to irreducible $V$} \label{S:red2irred}

If $V_\bR = V_1 \op V_2$ is reducible as a real representation, then the associated domain $D$ factors $D = D_1 \times D_2$ into the product of the domains $D_i$ for the $V_i$.  So without loss of generality we may assume that $V_\bR$ is irreducible.  The Schur lemma (and our hypothesis that \eqref{E:HrepLA} is faithful) implies
\begin{equation} \label{E:z}
  \tdim\,\fz \ \in \ \{ 0 , 1 \} \,.
\end{equation}

\begin{remark}\label{R:z}
Note that $\fz = \tspan\{\ttE'\}$, so that $\fg = \fg^\tss$ if and only if $\ttE'=0$.
\end{remark}

Given an irreducible real representation $V_\bR$ there exists a (unique) irreducible representation $U$ of $G_\bC$ such that one of the following holds:
\begin{equation} \label{E:VvU}
  V_\bR \ot \bC \ = \ 
  \left\{ \begin{array}{ll}
  U \tand U = U^* & (\hbox{$U$ is \emph{real} w.r.t.~$\fg_\bR$})\,, \\
  U\op U^* \tand U \not= U^* & (\hbox{$U$ is \emph{complex} w.r.t.~$\fg_\bR$})\,,\\
  U\op U^* \tand U = U^* & (\hbox{$U$ is \emph{quaternionic} w.r.t.~$\fg_\bR$})\,.
  \end{array} \right.
\end{equation}
Let $\mu , \m^* \in \fh^*$ denote the highest weights of $U$ and $U^*$ respectively.  When we wish to emphasize the highest weight of $U$, we will write $U = U_\mu$. 

\begin{remark}[Period domains]
In this case of Remark \ref{R:pd}, we have $V_\bC = U_{\w_1}$, with $\mu = \w_1$ the first fundamental weight.
\end{remark}

\begin{remark} \label{R:mu}
It follows from Remark \ref{R:z} that action of the center $\fz \subset \fg_\bR$ on $V_\bR$ is determined by the action of $\ttE'$.  The latter acts on $U$ by scalar multiplication by $c = \mu(\ttE') \in \bQ$.  In particular, $\fz \not=0$ if and only if $c\not=0$.  Moreover, $\ttE'$ necessarily acts on the dual by $-c = \mu^*(\ttE')$.  So $\mu \not= \mu^*$, and $U$ is complex with respect to $\fg_\bR$ whenever $\fg_\bR$ has a nontrivial center ($\fz \not=0$).
\end{remark}

\subsection{Real, complex and quaternionic representations} \label{S:rcq}

Note that $U$ is complex if and only if $\mu \not=\mu^*$.  By Remark \ref{R:mu}, this is always then case when $\fz \not=0$; equivalently, $\fg \not=\fg^\tss$.  When $\fz = 0$ (equivalently, $\fg = \fg^\tss$ is semisimple) the real and quaternionic representations may be distinguished as follows.   Recall the conventions of \S\ref{S:norm-ge}, and let $\{\ttA^1,\ldots,\ttA^r\} \subset \fh^\tss$ be the basis dual to the simple roots $\{\a_1,\ldots,\a_r\}$.  Then 
\[
  \ttE^\tss_\phi \ = \ \sum \a_i(\ttE_\phi) \ttA^i \,,
  \quad \hbox{with}\quad \a_i(\ttE_\phi) \in \{0,1\} \,.
\]
Define
\[
  \ttT_\phi \ := \ 2 \sum_{ \a_i(\ttE^\tss_\phi) = 0} \ttA^i \,.
\]
If $\mu = \mu^*$, then $U$ is real if and only if $\mu(\ttT_\phi)$ is even, and is quaternionic if and only if $\mu(\ttT_\phi)$ is odd \cite{MR2918237}.

\subsection{Eigenvalues and level of the Hodge structure} 

Set 
\[
  m \ := \ \mu(\ttE_\phi) \tand
  m^* \ := \ \mu^*(\ttE_\phi)\,.
\]
Then the nontrivial eigenvalues of $\ttE_\phi$ on $U$ are 
\[
  \{ m \,,\, m-1 \,,\, m-2 \,, \ldots ,\, 2-m^* \,,\, 1-m^* \,,\, -m^*\} \,.
\]
Equation \eqref{E:V} implies 
\[
  2m\,,\ 2m^* \ \in \ \bZ \,.
\]
The Hodge structure $\phi$ on $V_\bR$ is of level 
\[
  n \ = \ 2 \max\{ m,m^*\} \,.
\]

\subsection{Reductive versus semisimple} \label{S:rvss}

Let $(\fg_\bC,\ttE_\phi,\mu)$ be a triple underlying a Hodge representation; $\fg_\bC$ is a complex reductive Lie algebra, $\ttE_\phi \in \fg_\bC$ is a grading element (determining a real form $\fg_\bR$, \S\ref{S:gerf}), and $\mu$ is the highest weight of an irreducible $\fg_\bC$--module $U = U_\mu$.  The purpose of this section is to observe that such triples are equivalent to tuples $(\fg_\bC^\tss,\ttE^\tss_\phi,\mu^\tss,c)$ with $\fg^\tss_\bC$ a complex semisimple Lie algebra, $\ttE^\tss_\phi \in \fg_\bC^\tss$ a grading element, $\mu^\tss$ the highest weight of an irreducible $\fg_\bC^\tss$--module, and $c \in \bQ$.

Recall the notations of \S\ref{S:norm-ge}.  As discussed in Remark \ref{R:mu}, the central factor $\ttE'_\phi$ acts on the irreducible $U$ by a scalar 
\[
  c \ = \ \mu(\ttE_\phi') \ \in \ \bQ \,,
\]
and on $U^*$ by $-c$.  It follows that $(\fg_\bC,\ttE_\phi,\mu)$ and $(\fg^\tss_\bC,\ttE^\tss_\phi , \mu^\tss, c)$ carry the same data.  (Here $\mu^\tss = \left.\mu\right|_{\fh^\tss}$ is the highest weight of $U$ as a $\fg_\bC^\tss$--module.)  

As noted in Remark \ref{R:mu}, $\fg_\bC = \fg^\tss_\bC$ is semisimple if and only if $c=0$.

The remainder of this section is devoted to discussing the relationship between the $\ttE^\tss_\phi$--eigenspace decomposition of $U$ and the Hodge decomposition (\S\ref{S:basics}) of $V_\bR$.

Let 
\begin{equation}\label{E:ssd}
  U \ = \ U_{\mu(\ttE^\tss_\phi)} \,\op\cdots\op\, U_{-\mu^*(\ttE^\tss_\phi)}
\end{equation}
be the $\ttE^\tss_\phi$--eigenspace decomposition of $U$.  We have
\[
  m \ = \ \mu(\ttE_\phi) \ = \ \mu(\ttE^\tss_\phi) \,+\, c \,.
\]  
Likewise, the $\ttE^\tss_\phi$--eigenspace decomposition of $U^*$ is
\[
  U^* \ = \ U^*_{\mu^*(\ttE^\tss_\phi)} \,\op\cdots\op\, U^*_{-\mu(\ttE^\tss_\phi)} \,.
\]
It is a general fact from representation theory that $\mu(\ttE^\tss_\phi)$ and $-\mu^*(\ttE^\tss_\phi)$ are both elements of $\bQ$, and any two nontrivial $\ttE^\tss_\phi$--eigenvalues of $U$ differ by an integer.  
\begin{a_list}
\item \label{i:real}
If $U_\mu$ is real, then $\mu = \mu^*$ and $V_\bC = U_\mu$ imply that $c=0$ and 
\[
  V^{p,q} \ = \ U_{(p-q)/2} \,.
\]
(In this case, we have $\fz = 0$.)
\item \label{i:cpx-quat}
If $U$ is complex or quaternionic, then 
\[
  V^{p,q} \ = \ U_{(p-q)/2-c} \,\op\, U^*_{(p-q)/2+c} \,.
\]
\end{a_list}

\begin{remark}\label{R:ev}
From \eqref{E:ssd}, we see that the number of nontrivial $\ttE$--eigenvalues for $U_\mu$ is precisely $e(\mu,\ttE) = (\mu+\mu^*)(\ttE) + 1$.  By \eqref{E:V} and \eqref{E:VvU}, we have $e(\mu,\ttE) \le n+1$.  And by Remark \ref{R:phinot0}, $e(\mu,\ttE) \ge 2$.  Thus
\[
  2 \ \le \ e(\mu,\ttE) \,=\, (\mu+\mu^*)(\ttE) + 1 \ \le \ n+1 \,.
\]
\end{remark}

\section{Identification of Hodge domains: general strategy}\label{S:gen}

\subsection{Main result}

Given a complex semisimple Lie algebra $\fg_\bC$ with Cartan subalgebra $\fh \subset \fg_\bC$, and irreducible $\fg_\bC$--representation $U$ and a rational number $c \in \bQ$, let $\ttE' = c\,\tId \in \tEnd(U)$ be the operator acting on $U$ by scalar multiplication.  We specify that $\ttE' = -c\,\tId \in \tEnd(U^*)$ act by $-c$ on the dual representation.  Then 
\[
   \tilde \fg_\bC \ = \ \fg_\bC \op \tspan_\bC\{ \ttE'\}
\] 
is a reductive Lie algebra (semisimple if $c=0$), with semisimple factor $\fg_\bC$ and center $\fz$ spanned by $\ttE'$.  (We are essentially making a change of notation here, replacing the reductive/semisimple pair $\fg_\bC$, $\fg_\bC^\tss$ of the previous sections with (possibly) reductive/semisimple pair $\tilde\fg_\bC$, $\fg_\bC$.  This is done for notational simplicity: it is cleaner to drop the ${}^\tss$ superscript.)

The upshot of the discussions in \S\S\ref{S:ge}--\ref{S:rvss} is

\begin{theorem}[{Green--Griffiths--Kerr \cite{MR2918237}}] \label{T:strategy}
In order to identify the Hodge representations \eqref{E:Hrep} with specified Hodge numbers $\bh = (h^{n,0},\ldots,h^{0,n})$, it suffices to identify tuples $(\fg_\bC,\ttE,\mu,c)$ consisting of a complex semisimple Lie algebra $\fg_\bC$, a grading element $\ttE \in \fh \subset \fg_\bC$ \emph{(as in \S\ref{S:gerf} and \S\ref{S:norm-ge})}, the highest weight $\mu \in \fh^*$ of an a irreducible $\fg_\bC$--module $U$, and $c \in \bQ$ that satisfy the following conditions: $m := \mu(\ttE) + c \in \half \bZ$, and the irreducible representation $V_\bR$ \emph{(\S\ref{S:red2irred})} of the real form $\fg_\bR$ determined by $\ttE$ \emph{(\S\ref{S:gerf})} has $(\ttE + \ttE')$--eigenspace decomposition of the form \eqref{E:V} with $\tdim\,V_{(p-q)/2} = h^{p,q}$.
\end{theorem}

\begin{example}[Period domains] \label{eg:pd}
The domain $D_\phi$ is a period domain if and only if the tuple $(\fg_\bC,\ttE,\mu,c)$ is of one of the following two forms:
\begin{i_list}
\item
$\fg_\bC = \fsp_{2r}\bC$; $\mu = \w_1$, so that $U = \bC^{2r}$ is the standard representation; $\a_r(\ttE) = 1$ and $c=0$.
\item 
$\fg_\bC = \fso_m\bC$; $\mu = \w_1$, so that $U = \bC^{m}$ is the standard representation; and $c=0$.  If $m = 2r$ is even, then we also have $(\a_{r-1}+\a_r)(\ttE) \in \{0,2\}$.
\end{i_list}
\end{example}

\begin{example}[Weight $n=1$]\label{eg:wt1}
The weight $n=1$ Hodge representations well understood \cite{MR546620, MR2192012}.  The corresponding tuples $(\fg_\bC , \ttE , \mu , c)$ are
\begin{i_list}
\item
$(\fsp_{2r}\bC , \ttA^r , \w_1 , 0 )$, with $\tilde\fg_\bR = \fg_\bR = \fsp_{2r}\bR$.  The corresponding Hodge domain $D$ is the period domain $\cD$ parameterizing polarized Hodge structures with $\bh = (r,r)$.
\item
$(\fsl_{r+1}\bC , \ttA^1 , \w_i , \frac{i}{r+1}-\frac{1}{2})$,
with $\tilde\fg_\bR = \fg_\bR = \fsu(1,r)$ if $2i = r+1$, and $\tilde\fg_\bR = \fu(1,r)$ otherwise.
\item
$(\fsl_{a+b}\bC,\ttA^a , \w_1 , \half - \frac{b}{a+b})$, with 
$\tilde\fg_\bR = \fg_\bR = \fsu(a,a)$ if $a=b$, and $\tilde\fg_\bR = \fu(a,b)$ otherwise.
\item
$(\fso_{m+2}\bC , \ttA^1 , \w_r , 0)$, with $m+2 \in \{2r,2r+1\}$ and $\tilde\fg_\bR = \fg_\bR = \fso(2,m)$.
\item
$(\fso_{2r}\bC , \ttA^r , \w_1 , 0 )$, with $\tilde\fg_\bR = \fg_\bR = \fso^*(2r)$.
\end{i_list}
\end{example}

\begin{remark} \label{R:redundancy}
Note that the two tuples $(\fg_\bC,\ttE,\mu,c)$ and $(\fg_\bC,\ttE,\mu^*,-c)$ determine the same Hodge representation (\S\S\ref{S:red2irred} \& \ref{S:rvss}).
\end{remark}

\begin{remark} \label{R:irred} 
One consequence of Remark \ref{R:obvious} is that in any particular example -- that is, the case of a fixed period domain $\cD$ with specified Hodge numbers $\bh$ -- it suffices to identify the \emph{irreducible} Hodge domains $D$ with Hodge numbers $\bh' \le \bh$.  For example, in \S\ref{S:eg}, where we consider the case that $\bh = (2,h^{1,1},2)$, it will suffice to consider the two cases that $\bh' = (1,h,1)$ and $\bh' = (2,h,2)$ with $h \le h^{1,1}$.
\end{remark}

\subsection{Horizontal Hodge domains} \label{S:horiz}

Theorem \ref{T:strategy} identifies all the Hodge subdomains $D$ of the period domain $\cD_\bh$.  We are especially interested in the horizontal subdomains, which are necessarily Hermitian.  These are the domains that satisfy the infinitesimal period relation (IPR, a.k.a.~Griffiths' transversality).  These distinguished subdomains may be identified as follows.

It is a consequence of the normalization in \S\ref{S:norm-ge} that the Hodge subdomain $D \subset \cD$ is horizontal if and only if the induced Hodge decomposition \eqref{E:ind} is of the form
\begin{equation}\label{E:herm}
  \fg_\bC \ = \ \fg^{1,-1}_\phi \,\op\,\fg^{0,0}_\phi \,\op\, \fg^{-1,1}_\phi \,;
\end{equation}
that is, $\fg^{\ell,-\ell}_\phi = 0$ for all $|\ell| \ge 2$, cf.~\cite{MR2532439}, \cite[\S\S2--3]{MR3217458}.  This is a condition on the grading element:
\[
  \tilde\a(\ttE_\phi) \ = \ 1 \,,
\]
where $\tilde\a$ is the highest root.  All such domains are necessarily Hermitian symmetric.  

For the simple, complex Lie groups $\fg_\bC$ the set of all such grading elements (see \S\ref{S:rcq} for notation), the corresponding compact duals $\check D$, the real forms $\fg_\bR$, and the maximal compact subalgebra $\fk \subset \fg_\bR$ are listed in Table \ref{t:herm}.
\begin{table}[b]
\caption{Data underlying irreducible Hermitian symmetric Hodge domains}
\[
\begin{array}{|ccccc|}
  \hline
  \fg_\bC & \ttE & \check D = G_\bC/P_\ttE & \fg_\bR & \fk \\ \hline
  \fsl(a+b,\bC) & \ttA^a & \tGr(a,\bC^{a+b}) & \fsu(a,b) 
  & \mathfrak{s}(\fu(a) \op \fu(b)) \\
  \fso(d+2,\bC) & \ttA^1 & \cQ^d & \fso(2,d) 
  & \mathfrak{s}(\fo(2) \op \fo(d)) \\
  \fsp(2r,\bC) & \ttA^r & \tGr^Q(r,\bC^{2r}) & \fsp(2r,\bR) 
  & \fu(r) \\
  \fso(2r,\bC) & \ttA^r & \hbox{Spinor variety} & \fso^*(2r) 
  & \fu(r) \\
  \fe_6 & \ttA^6 & \hbox{Cayley plane} & \mathrm{E III} 
  & \fso(10) \op \bR \\
  \fe_7 & \ttA^7 & \hbox{Freudenthal variety} & \mathrm{E VII} 
  & \fe_6 \op \bR \,.
  \\ \hline
\end{array}
\]
\label{t:herm}
\end{table}
Here $\tGr(a,\bC^{a+b})$ is the grassmannian of $a$--planes in $\bC^{a+b}$, $\cQ^d \subset \bP^{d+1}$ is the quadric hypersurface, and $\tGr^Q(r,\bC^{2r})$ is the Lagrangian grassmannian of $Q$--isotropic $r$--planes in $\bC^{2r}$.  The following proposition is immediate and well-known.

\begin{proposition} \label{P:horiz}
If $(\fg_\bC,\ttE,\mu,c)$ is a tuple indexing a Hodge representation \eqref{E:Hrep} \emph{(cf.~Theorem \ref{T:strategy})}, then the resulting Hodge domain $D_\phi$ is horizontal if and only if $(\fg_\bC,\ttE)$ is a sum of those pairs listed in Table \ref{t:herm}.
\end{proposition}

In general the Hodge domains $D \subset \check D$ are cut out by nondegeneracy conditions defined by a Hermitian form $\sH$.  For example, in the case of period domains, the compact dual essentially encodes the first Hodge--Riemann bilinear relation, and the second Hodge--Riemann bilinear relation is the nondegeneracy condition cutting out $D$.  To illustrate this, we describe the Hodge domains for the first three rows of Table \ref{t:herm}.
\begin{numlist}
\item
In the case of $\check D = \tGr(a,\bC^{a+b})$, we note that $\bC^{a+b}$ has an underlying real structure, and we fix a nondegenerate Hermitian form $\sH$ on $\bC^{a+b}$ of signature $(a,b)$.  Then 
\[
  D \ = \ \left\{ E\in\tGr(a,\bC^{a+b}) \ \left| \ 
  \left.\sH\right|_E \hbox{ is pos def} \right.\right\} \,.
\]
\item
In the case that $\check D = \cQ^d = \tGr^Q(1,\bC^{d+2})$ we define a Hermitian form $\sH$ on $\bC^{d+2}$ by $\sH(u,v) = -Q(u,\bar v)$.  Then 
\[
  D \ = \ \left\{ E\in\tGr^Q(1,\bC^{d+2}) \ \left| \ 
  \left.\sH\right|_E \hbox{ is pos def} \right.\right\} \,.
\]
\item
In the case that $\check D = \tGr^Q(r,\bC^{2r})$ we define a Hermitian form $\sH$ on $\bC^{2r}$ by $\sH(u,v) = \bi \,Q(u,\bar v)$.  Then 
\[
  D \ = \ \left\{ E\in\tGr^Q(r,\bC^{2r}) \ \left| \ 
  \left.\sH\right|_E \hbox{ is pos def} \right.\right\} \,.
\]
\end{numlist}

\section{Example: Hodge domains for level $2$ Hodge structures} \label{S:eg}

The purpose of this section is to illustrate the application of the strategy outlined in \S\ref{S:gen} in the case that  $\cD = \cD_\bh$ is the period domain parameterizing $Q$--polarized, (effective) weight two Hodge structures on $V_\bR$ with Hodge numbers 
\[
  \bh \ = \ (h^{2,0},h^{1,1},h^{0,2}) \ = \ (2,h^{1,1},2) \,.
\]
Equivalently, $\varphi \in \cD$ parameterizes Hodge decompositions 
\[
  V_\bC \ = \ V^{2,0} \op V^{1,1} \op V^{0,2} \,,
\]
with
\[
  \tdim_\bC\,V^{2,0} \ = \ 2 \ = \ \tdim_\bC\,V^{0,2} \,.
\]
(We assume throughout that $h^{1,1} = \tdim_\bC\,V^{1,1} \not=0$.)  Geometrically such Hodge structures arise when studying smooth projective surfaces with $p_g=2$.  

We have 
\[  
  \cG_\bR \ = \ \tAut(V_\bR,Q) \ = \ \tO(h^{1,1},4) \,.
\]
As discussed in \S\ref{S:approach} it suffices to identify the irreducible Hodge representations \eqref{E:Hrep} with either $\bh_\phi = (1,h,1)$ or $\bh_\phi = (2,h,2)$, and $h \le h^{1,1}$.  (Each such Hodge representation corresponds to a Hodge subdomain $D = G_\bR \cdot \phi$ of the period domain $\cD = \cD_{\bh_\phi}$ parameterizing $Q$--polarized Hodge structures on $V_\bR$ with Hodge numbers $\bh_\phi$.)   The analysis decomposes into three parts:
\begin{A_list}
\item
We begin with the simplifying assumptions that $\fg_\bC$ is simple and that $D$ is horizontal.  This has the strong computational advantage that we may take the grading element $\ttE$ to be as listed in Table \ref{t:herm}.  The resulting domains are enumerated in Theorems \ref{T:pg=1} and \ref{T:pg=2}.  
\item 
Continuing to assume that $\fg_\bC$ is simple, we turn to the case that horizontality fails; the domains are enumerated in Theorem \ref{T:notH}.
\item
Finally we consider in Theorem \ref{T:ss} the case that $\fg_\bC$ is semisimple (but not simple).
\end{A_list}
Together Theorems \ref{T:pg=1}, \ref{T:pg=2}, \ref{T:notH} and \ref{T:ss} give a complete list of the irreducible Hodge representations \eqref{E:Hrep} with $\bh_\phi \le \bh = (2,h^{1,1},2)$.

\begin{theorem} \label{T:pg=1}
The irreducible Hodge representations \eqref{E:Hrep} with $\fg_\bC$ simple, $\bh_\phi = (1,h,1)$ and horizontal Hodge domain $D\subset \cD_{(1,h,1)}$ are given by the following tuples $(\fg_\bC,\ttE,\mu,c)$:
\begin{i_list_emph}
\item \label{i:1pd}
Period domains: $(\fso(h+2,\bC) , \ttA^1 , \w_1 , 0)$, with $\bh_\phi = (1,h,1)$.
\item \label{i:1ghd-i}
Grassmannian Hodge domains: 
both tuples 
\[
  (\fsl(1+r,\bC),\ttA^1,\w_r,-1/(r+1)) \tand
  (\fsl(1+r,\bC),\ttA^r,\w_1,-1/(r+1))
\]
yield Hodge representations with $\bh_\phi = (1,2r,1)$.
\end{i_list_emph}
\end{theorem}

\begin{remark}[Geometric realizations]
Pearlstein and Zhang \cite{MR3914747} have exhibited geometric realizations of $G_\varphi = G_1 \times G_2$ with $G_i$ one of $\tSO(2,h_i)$ or $\tU(1,r_i)$, corresponding to the two cases/factors of Theorem \ref{T:pg=1}.
\end{remark}

\begin{theorem} \label{T:pg=2}
The irreducible Hodge representations \eqref{E:Hrep} with $\fg_\bC$ simple, $\bh_\phi = (2,h,2)$ and horizontal Hodge domain $D\subset \cD_{(2,h,2)}$ are all grassmannian Hodge domains \emph{(corresponding to the first row of Table \ref{t:herm})}, and are given by the following tuples $(\fsl(a+b,\bC) ,\ttA^a,\mu,c)$:
\begin{i_list_emph}
\item \label{i:2ghd-i}
The tuples 
\[
  (\fsl(3,\bC),\ttA^1,\w_2,2/3) \tand
  (\fsl(3,\bC),\ttA^1,\w_1,-2/3)
\]
yield Hodge representations with $\bh_\phi = (2,2,2)$.
\item \label{i:2ghd}
The tuple $(\fsl(r+1,\bC),\ttA^2,\w_1,2/(r+1))$
yields a Hodge representation with $\bh_\phi = (2,2r-2,2)$.
\end{i_list_emph}
\end{theorem}


\begin{theorem} \label{T:notH}
The irreducible Hodge representations \eqref{E:Hrep} with $\fg_\bC$ simple and $\bh_\phi \le \bh = (2,h^{1,1},2)$, for which the Hodge domain $D\subset \cD_\bh$ is \emph{not} horizontal are given by the following tuples $(\fg_\bC,\ttE,\mu,c)$:
\begin{i_list_emph}
\item \label{i:not-pd}
Period domains: $(\fso(h+4,\bC) , \ttA^2 , \w_1 , 0)$, with $\bh_\phi = (2,h,2)$.
\item 
Special Linear contact domains: $(\fsl(r+1,\bC) , \ttA^1+\ttA^r , \w_1 , 0)$, 
with $\bh_\phi = (2,2r-2,2)$.
\item 
Special Linear contact domains: $(\fsl(4,\bC) , \ttA^1+\ttA^3 , \w_2 , 0)$, with $\bh_\phi = (2,2,2)$.
\item \label{i:2spin}
Spinor contact domains: $(\fso(5,\bC),\ttA^2,\w_2,0)$ and $(\fso(7,\bC),\ttA^2,\w_3,0)$, both with $\bh_\phi = (2,4,2)$.  \emph{(The first is quaternionic, the second is real.)}
\item \label{i:2sym}
Symplectic contact domains: $(\fsp(2r,\bC),\ttA^1,\w_1,0)$, with $\bh_\phi = (2,4(r-1),2)$.
\item \label{i:g2}
Exceptional contact domains: $(\fg_2,\ttA^2,\w_1,0)$ with $\bh_\phi = (2,3,2)$.
\end{i_list_emph}
\end{theorem}

\noindent
See \S\ref{S:notH} for further discussion of the domains $D$ appearing in Theorem \ref{T:notH} as homogeneous spaces.

\begin{remark}
The Spinor contact domain given by the tuple $(\fso(5,\bC),\ttA^2,\w_2,0)$ in Theorem \ref{T:pg=2}\emph{\ref{i:2spin}} is a special case of Theorem \ref{T:pg=2}\emph{\ref{i:2sym}} under the isomorphism $\fso(5,\bC) \simeq \fsp(4,\bC)$.
\end{remark}


\begin{theorem}\label{T:ss}
The irreducible Hodge representations \eqref{E:Hrep} with $\bh_\phi \le \bh = (2,h^{1,1},2)$ and $\fg_\bC$ semisimple (but not simple) are given by:
\begin{i_list_emph}
\item
$\fg_\bC = \fsl_2\bC \op \fsl_2\bC$ acting on $U = \bC^2 \ot \bC^2$ with $\ttE = \ttA^1 + \ttA^2$; and
\item
$\fg_\bC = \fsl_2\bC \op \fsp_4\bC$ acting on $U = \bC^2 \ot \bC^4$ with $\ttE = \ttA^1 + \ttA^3$.
\end{i_list_emph}
Each of these Hodge representations is real (implying $c=0$).  The Hodge numbers are $\bh = (1,2,1)$ and $\bh = (2,4,2)$, respectively.
\end{theorem}

The remainder of this section is devoted to the proofs of these theorems.  The general argument is outlined in \S\ref{S:prf0}.  Theorems \ref{T:pg=1} and \ref{T:pg=2} are proved simultaneously in \S\S\ref{S:prf-ghd}--\ref{S:prf-fhd}; Theorem \ref{T:notH} is proved in \S\ref{S:notH}; and Theorem \ref{T:ss} is proved in \S\ref{S:ss}.

\subsection{Outline of the arguments} \label{S:prf0}

The proofs of Theorems \ref{T:pg=1} and \ref{T:pg=2} proceed by considering each of the cases listed in Table \ref{t:herm}.  Given the pair $(\fg_\bC,\ttE)$ it suffices to determine when there exists an irreducible $\fg_\bC$--representation $U_\mu$ of highest weight $\mu \in \fh^*$, and $c \in \bQ$ satisfying the conditions of Theorem \ref{T:strategy} for the specified $\bh_\phi$. First note that $U_\mu$ has either two or three nontrivial $\ttE$--eigenvalues; equivalently (Remark \ref{R:ev}), 
\begin{equation}\label{E:wt-rest}
  (\mu + \mu^*)(\ttE) \ \in \ \{ 1,2 \} \,.
\end{equation}
This gives us the following three possibilities (cf.~\S\S\ref{S:rcq} and \ref{S:rvss}): 
\begin{a_list}
\item \label{a:R3}
If $U_\mu$ is real, then $c=0$ (\S\ref{S:rvss}\ref{i:real}) and it is necessary and sufficient that the $\ttE$--eigenspace decomposition (equivalently, the Hodge decomposition) of $V_\bC = U_\mu$ be 
\[
  V^{2,0} \op V^{1,1} \op V^{0,2} \ = \ U_1 \op U_0 \op U_{-1} \,,
\]
with $\tdim\,U_{\pm1} = h^{2,0}_\phi \in \{1,2\}$.  In particular, $\mu(\ttE) = 1$.
\item \label{a:CQ}
If $U_\mu$ is complex or quaternionic, so that $V_\bC = U_\mu \op U_\mu^*$, and there are three nontrivial $\ttE$--eigenvalues, so that the $\ttE$--eigenspace decompositions are 
\begin{eqnarray*}
  U_\mu & = &  U_{\mu(\ttE)} \op U_{\mu(\ttE)-1} \op U_{\mu(\ttE)-2} \\
  U_\mu^* & = & U_{2-\mu(\ttE)} \op U_{1-\mu(\ttE)} \op U_{-\mu(\ttE)} \,.
\end{eqnarray*}
Then we are looking for $c \in \bQ$ so that
\[
  \begin{array}{c|ccc}
  V_\bC & V^{2,0} & V^{1,1} & V^{0,2} \\ \hline
  U_\mu & U_{\mu(\ttE)} & U_{\mu(\ttE)-1} & U_{\mu(\ttE)-2} \\
  U_\mu^* & U_{2-\mu(\ttE)} & U_{1-\mu(\ttE)} & U_{-\mu(\ttE)} \,.
  \end{array}
\]
Equivalently, $\mu(\ttE)+c = 1$ and $2-\mu(\ttE)-c=1$.  That is,
\[
  c \ = \ 1-\mu(\ttE) \,.
\]
Note that each of the eigenspaces $U_{\pm\mu(\ttE)}$ and $U_{\pm(\mu(\ttE)-2)}$ must have dimension one, and we have $h^{2,0} = 2$.  (In particular, this case will not appear in Theorem \ref{T:pg=1}.) 
\item \label{a:C2}
Suppose $U_\mu$ is complex, so that $V_\bC = U_\mu \op U_\mu^*$, and there are two nontrivial $\ttE$--eigenvalues, so that the $\ttE$--eigenspace decompositions are 
\begin{eqnarray*}
  U_\mu & = &  U_{\mu(\ttE)} \op U_{\mu(\ttE)-1} \\
  U_\mu^* & = & U^*_{1-\mu(\ttE)} \op U^*_{-\mu(\ttE)} \,.
\end{eqnarray*}
We are looking for $c \in \bQ$ so that either 
\[
  \begin{array}{c|ccc}
  V_\bC & V^{2,0} & V^{1,1} & V^{0,2} \\ \hline
  U_\mu & U_{\mu(\ttE)} & U_{\mu(\ttE)-1} &  \\
  U_\mu^* &  & U^*_{1-\mu(\ttE)} & U^*_{-\mu(\ttE)} 
  \end{array} \,.
\]
or
\[
  \begin{array}{c|ccc}
  V_\bC & V^{2,0} & V^{1,1} & V^{0,2} \\ \hline
  U_\mu & & U_{\mu(\ttE)} & U_{\mu(\ttE)-1} \\
  U_\mu^* & U^*_{1-\mu(\ttE)} & U^*_{-\mu(\ttE)} &  
  \end{array} \,.
\]
Equivalently, either 
\[
  1 \ = \ \mu(\ttE) + c \tand \tdim_\bC\,U_{\mu(\ttE)} 
  = h^{2,0}_\phi \in \{1,2\} \,,
\]
or 
\[
  \mu(\ttE) \ = \  - c \tand \tdim_\bC\,U^*_{1-\mu(\ttE)} 
  = h^{2,0}_\phi \in \{1,2\} \,.
\]
\end{a_list}

\smallskip

\noindent The proofs of Theorems \ref{T:pg=1} and \ref{T:pg=2} now proceed by applying the observations of this section to each pair $(\fg_\bC,\ttE)$ corresponding to a row of Table \ref{t:herm}.

\smallskip

We now turn to the simultaneous proofs of Theorems \ref{T:pg=1} and \ref{T:pg=2} in \S\S\ref{S:prf-ghd}--\ref{S:prf-fhd}, followed by the proofs of Theorems \ref{T:notH} and \ref{T:ss} in \S\ref{S:notH} and \S\ref{S:ss}, respectively.

\subsection{Grassmannian Hodge domains} \label{S:prf-ghd}

We begin with the first row of Table \ref{t:herm} and the pair $(\fg_\bC,\ttE) = (\fsl(a+b,\bC) \,,\, \ttA^a)$.\footnote{Despite what the reader might anticipate, this case/row is the most tedious and painstaking to work through.  This is essentially due to the numerically more complicated relationship between the roots (dual to the basis $\ttA^a$ for the grading elements) and the weights (i.e.~the complexity in the Cartan matrix) for $\fg_\bC = \fsl(a+b,\bC)$.  The other cases \S\S\ref{S:prf-qhd}--\ref{S:prf-fhd} are easier to analyze.}  

The standard representation $U_{\w_1} = \bC^{a+b}$ of $\fg_\bC = \fsl_{a+b}\bC$ admits a decomposition $\bC^{a+b} = A \op B$ with $\tdim\,A = a$ and $\tdim\,B = b$ and such that $A$ is an eigenspace of $\ttE$ with eigenvalue $b/(a+b)$, and $B$ is an eigenspace with eigenvalue $-a/(a+b)$.  It will be helpful to note that the $\ttE$--eigenspace decomposition of $\tw^i\bC^{a+b}$ is 
\begin{equation}\label{E:epA}
  \tw^i(A \op B) \ = \ \bigoplus_{\a+\b=i} 
  (\tw^\a A) \ot (\tw^\b B)\,.
\end{equation}
Fix bases $\{e_1,\ldots,e_a\}$ and $\{e_{a+1},\ldots,e_{a+b}\}$ of $A$ and $B$, respectively.

We assume throughout \S\ref{S:prf-ghd} that $a+b = i+j = k+\ell = r+1$.  Consulting \S\ref{S:prf0} and \S\ref{S:ghd}, we see that the pair $(\mu,\ttE = \ttA^a)$ must be one of the following:
\begin{i_list}
\item \label{i:A1}
$a=1$ and $\mu = \w_i$, any $1 \le i \le r$;
\item \label{i:A2}
$a=1$ and $\mu = \w_i+\w_k$, any $1\le i,k \le r$;
\item \label{i:A3}
$a=2$ and $\mu = \w_i$, any $2 \le i \le r-1$;
\item \label{i:A4}
$\mu \in \{ \w_1 , 2 \w_1\}$, any $2 \le a \le r-1$;
\item \label{i:A5}
$\mu = \w_2$ and any $2 \le a \le r-1$.
\end{i_list}
(This list suppresses some cases that are essentially symmetric with those already listed.  For example $\ttE = \ttA^r$ and $\mu = \w_i$ is symmetric with \ref{i:A1}.)  We proceed to consider each of these five cases.

\smallskip

{\bf \ref{i:A1}} 
Consulting \eqref{E:epA} we see that 
\[
  U_{\w_i} \ = \ \tw^i(A \op B) \ = \ 
  \left( A \ot \tw^{i-1}B \right) \ \op \ \left(\tw^i B \right) \,.
\]
These eigenspaces have dimensions $\left( \binom{r}{i-1} , \binom{r}{i} \right)$.  In order to realize a Hodge representation with $h^{2,0}_\phi \in \{1,2\}$, one of these dimensions must be $1$ or $2$.  

	The first dimension will be one if and only if $r=1$ (which forces $i=1$).  But in this case the representation $U_\mu$ is real, and so the resulting Hodge representation will be weight $n=1$, not the desired weight $n=2$.
	
	The second dimension will be one if and only if $i=r$.  Then the dimensions of the $\ttE$--eigenspaces of $U_{\w_r}$ and $U_{\w_r}^* = U_{\w_1}$ are $(r,1)$ and $(1,r)$, respectively.  The eigenvalue for $\tw^rB \subset U_{\w_r}$ is $-r/(r+1)$.  So setting $c = -1/(r+1)$ gives us a Hodge representation with eigenvalues $\bh_\phi = (1,2r,1)$, yielding Theorem \ref{T:pg=1}\emph{\ref{i:1ghd-i}}.
	
	The first dimension will be two if and only if $i=r=2$.  In this case the dimensions are $(2,1)$, and $U_\mu = U_{\w_2}$ is complex with $U_\mu^* = U_{\w_1} = \bC^3$.  The $\ttE$--eigenspaces of $U_{\mu}^*$ have dimensions $(1,2)$.  We have $\mu(\ttE) = 1/3$.  Setting $c=-1/3$ yields a special case of Theorem \ref{T:pg=1}\emph{\ref{i:1ghd-i}}, and setting $c=2/3$ yields a special case of Theorem \ref{T:pg=2}\emph{\ref{i:2ghd}} (Remark \ref{R:redundancy}).
	
	The second dimension will be two if and only if $r=2$ and $i=1$.  In this case the dimensions are $(1,2)$, and $U_\mu = U_{\w_1} = \bC^3$ is complex with $U_\mu^* = U_{\w_2} = \tw^2\bC^3$.  The $\ttE$--eigenspaces of $U_\mu^*$ have dimensions $(2,1)$.  We have $\mu(\ttE) = 2/3$.  Setting $c=1/3$ yields a special case of Theorem \ref{T:pg=1}\emph{\ref{i:1ghd-i}} (Remark \ref{R:redundancy}).  Setting $c=-2/3$ yields a special case of Theorem \ref{T:pg=2}\emph{\ref{i:2ghd}}.
	
\smallskip

{\bf \ref{i:A2}} 
We have $U_{\w_i+\w_k} \subset (\tw^i\bC^{r+1}) \ot (\tw^k\bC^{r+1})$, with the latter having three distinct $\ttE$--eigenspaces
\begin{eqnarray*}
  (\tw^i\bC^{r+1}) \ot (\tw^k\bC^{r+1}) \
  & = & 
  \left( A \ot A \ot (\tw^{i-1}B) \ot (\tw^{k-1} B) \right)  \\
  & & \op \ 
  \left\{ \begin{array}{c}
		\left( A \ot (\tw^{i-1}B) \ot (\tw^kB) \right) \\ 
		\left( A \ot (\tw^{i}B) \ot (\tw^{k-1}B) \right)
  \end{array} \right. \\
  & & \op \ 
  \left((\tw^{i}B) \ot (\tw^{k} B) \right)
\end{eqnarray*}
The product $(e_1 \wedge\cdots\wedge e_i) \ot (e_1 \wedge\cdots\wedge e_k) \in A \ot A \ot (\tw^{i-1}B) \ot (\tw^{k-1} B)$ is a highest weight vector of $U_{\w_i+\w_k}$.  Without loss of generality $i \le k$.  The products
\[
  (e_1 \wedge\cdots\wedge e_i) \ot (e_1 \wedge\cdots\wedge e_h) \,,\quad
  k \le h \le r+1 \,,
\] 
are all elements of the first eigenspace $U_{\mu(\ttE)} \subset A \ot A \ot (\tw^{i-1}B) \ot (\tw^{k-1} B)$.  Because this eigenspace may have dimension at most $h^{2,0}_\phi \le 2$, we see that $k=r$ (and the eigenspace has dimension at least $2$).  Likewise
\[
  (e_1 \wedge\cdots\wedge e_h) \ot (e_1 \wedge\cdots\wedge e_r) \,,\quad
  i \le h \le r \,,
\] 
are also elements of this eigenspace; and dimension/Hodge number considerations again force $i=k=r$.  The representation $U_{2\w_r}$ is complex, unless $r=1$; if complex, then the associated Hodge representation has $h^{2,0}_\phi > 2$, which is too large.  So we must have $r=1$, in which case $V_\bC = U_{2\w_1} = \tSym^2\bC^2$ is real and we have $\bh_\phi = (1,1,1)$.  However, under the isomorphism $\fsl_2\bC \simeq \fso(3,\bC)$, this is a special case of Theorem \ref{T:pg=1}\emph{\ref{i:1pd}}. 

\smallskip

{\bf \ref{i:A3}} 
In this case we have $\ttE$--eigenspace decomposition
\[
  U_{\w_i} \ = \ \tw^i(\bC^{2+b}) \ = \ 
  \left( (\tw^2A) \ot (\tw^{i-2}B) \right) \ \op \ 
  \left( A \ot (\tw^{i-1}B) \right) \ \op \ 
  (\tw^{i}B) \,.
\]
The condition that the first and third eigenspaces $(\tw^2A) \ot (\tw^{i-2}B)$ and $\tw^iB$ have dimensions 1 or 2 forces $i=b=2$.  Then $U_\mu = U_{\w_1}$ is self-dual and real.  This is a special case of Theorem \ref{T:pg=1}\emph{\ref{i:1pd}} under the isomorphism $\fsl(4,\bC)\simeq \fso(6,\bC)$.  

\smallskip

{\bf \ref{i:A4}}  
If $\mu = \w_1$, then $U_\mu = \bC^{a+b} = A \op B$ is the standard representation.  Recalling the discussion at the beginning of this section we see that we must have either $a=2$ or $b = r+1-a = 2$.  Taking $c = 2/(r+1)$ if $a=2$, and $c=-2/(r+1)$ if $b=2$, yields $\bh_\phi = (2 , 2r-2 , 2)$ and Theorem \ref{T:pg=2}\emph{\ref{i:2ghd}} (Remark \ref{R:redundancy}).

If $\mu = 2\w_1$, then $U_\mu = \tSym^2\bC^{a+b} = (\tSym^2 A) \op (A\ot B) \op (\tSym^2 B)$.  In this case $\tdim_\bC\,\tSym^2A \ge 3 > h^{2,0}_\phi$ is too large.

\smallskip

{\bf \ref{i:A5}} 
If $\mu = \w_2$, then $U_\mu = \tw^2\bC^{a+b} = (\tw^2 A) \op (A \ot B) \op (\tw^2 B)$.  The first and third eigenspaces $\tw^2 A$ and $\tw^2 B$ are constrained to have dimension at most $h^{2,0}_\phi \le 2$.  This forces $a=b=2$.  In this case $U_\mu$ is real and we have Hodge numbers $(1,4,1)$.  This is the special case of Theorem \ref{T:pg=1}\emph{\ref{i:1pd}} that we encountered above in part \ref{i:A3} of the proof.

\subsection{Quadric hypersurface Hodge domains} \label{S:prf-qhd}

We next consider the second row of Table \ref{t:herm} and the pair $(\fg_\bC,\ttE) = (\fso(d+2,\bC) \,,\, \ttA^1)$.  Here we may assume that either $d = 3$ or $d \ge 5$ (else we are in the case considered in \S\ref{S:prf-ghd}).  

\subsubsection{Period domains} \label{S:prf-pd}

If $\mu = \w_1$, so that $U_\mu = \bC^{d+2}$ is the standard representation, and real with respect to $(\fg_\bC, \ttE)$, then $V_\bC = U_{\w_1}$ has eigenspace decomposition $\bC \op \bC^d \op \bC$ with eigenvalues $(1,d,1)$.  Of course, in this case the Hodge domain is the period domain $\cD$ parameterizing $Q$--polarized Hodge structures with Hodge numbers $\bh = (1,d,1)$.  

\subsubsection{Exterior powers} \label{S:prf-ep}

For the analysis that follows, it will be helpful to make the following observations about exterior powers of the standard representation.  Given $2 \le i \le r-1 \le \half d$, the representation $\tw^i \bC^{d+2}$ is real, defines a Hodge representation, and has $\ttE$--eigenspace decomposition 
\begin{eqnarray*}
 \tw^i(\bC \op \bC^d \op \bC) & = & 
 \left( \bC \ot (\tw^{i-1}\bC^d) \right) \\
 & & \op \ 
 \left(  (\bC \ot (\tw^{i-2}\bC^d) \ot \bC ) \ \op \ (\tw^i\bC^d) \right) \\
 & & \op \ \left( (\tw^{i-1}\bC^d) \ot \bC \right)  \,.
\end{eqnarray*}
The dimension $h^{2,0}_\phi$ of the first eigenspace $\bC \ot (\tw^{i-1}\bC^d)$ is $\binom{d}{i-1}$.  We have $h^{2,0}_\phi \in \{ 1,2 \}$ if and only if $i=2$ and $d=2$.  But we are assuming $d\ge 3$.
	
	\smallskip
	
	We assume $\mu \not= \w_1$ for the remainder of \S\ref{S:prf-qhd}.  (The case $\mu = \w_1$ is treated in \S\ref{S:prf-pd}.)  The representation theory of $\fg_\bC = \fso(d+2,\bC)$ depends on the parity of $d$; we begin with $d$ odd.

\subsubsection{The case of $d$ odd}

Assume $d \equiv 1$ mod $2$.  Consulting \eqref{E:wt-rest} and \S\ref{S:qhd}, we see that either $\mu = \w_i$ with $2 \le i \le r-1$, or $\mu \in \{ \w_r , 2\w_r \}$.  In the first case we have $U_{\w_i} = \tw^i\bC^{d+2}$, which is treated in \S\ref{S:prf-ep}.

\smallskip

$\bullet$  
The representation $U_{2\w_r} = \tw^r U_{\w_1} = \tw^r\bC^{d+2}$ has $\ttE$--eigenspace decomposition \[
  U_{2\w_r} \ = \ 
  \left( \bC \ot (\tw^{r-1} \bC^d) \right) \ \op \ 
  \tw^r\bC^d \ \op \ 
  \left( (\tw^{r-1} \bC^d) \ot \bC \right) \,.
\]
The resulting Hodge representation has $h^{2,0}_\phi \ge \tdim_\bC \tw^{r-1} \bC^d \ge 3$, which is too large.

\smallskip

$\bullet$ Likewise, the dimensions $(2^{r-1},2^{r-1})$ of the $\ttE$--eigenspaces in the spinor representation $U_{\w_r}$ are to large, unless $r=2$.  But in this case that representation is real, and the Hodge representation is of weight $1$ (\S\ref{S:prf0}\ref{a:R3}).

\subsubsection{The case of $d$ even}

Assume $d \equiv 0$ mod $2$.  Consulting \eqref{E:wt-rest} and \S\ref{S:qhd}, we see that either $\mu = \w_i$ with $2 \le i \le r-2$, or $\mu \in \{\w_{r-1} , \w_r\} \cup \{ 2\w_{r-1} , \w_{r-1}+\w_r, 2\w_r \}$.  In the first case we have $U_{\w_i} = \tw^i\bC^{d+2}$, which is treated in \S\ref{S:prf-ep}.  Likewise, $U_{\w_{r-1}+\w_r} = \tw^{r-1} \bC^{d+2}$ is treated in \S\ref{S:prf-ep}.

\smallskip

$\bullet$ 
The cases $\mu = \w_{r-1}$ and $\mu = \w_r$ are symmetric, so we treat $\mu = \w_r$ here.  The half-spin representation $U_{\w_r}$ decomposes into two $\ttE$--eigenspaces of dimensions $(2^{r-2},2^{r-2})$.  Since $r \ge 4$, these dimensions are too large to realize a Hodge representation (as in \S\ref{S:prf0}) with $h^{2,0}_\phi \in \{1,2\}$.

\smallskip

$\bullet$ 
Similarly the cases $\mu = 2\w_{r-1}$ and $\mu = 2\w_r$ are symmetric, and we treat $\mu = 2\w_r$ here.  We have $\tw^r \bC^{d+2} = \tw^r\bC^{2r} = U_{2\w_{r-1}} \op U_{2\w_r}$.  The representation $U_{2\w_r}$ decomposes into three $\ttE$--eigenspaces, the first and last of which have dimension $\half\binom{2r-2}{r-1}$.  Again, since $r \ge 4$, these dimensions are too large to realize a Hodge representation (as in \S\ref{S:prf0}) with $h^{2,0}_\phi \in \{1,2\}$.

\subsection{Lagrangian grassmannian Hodge domains} \label{S:prf-lghd}

Consider the third row of Table \ref{t:herm} and the pair $(\fg_\bC,\ttE) = (\fsp(2r,\bC) \,,\, \ttA^r)$.  Here we may assume $r\ge 3$ (else we are in the case considered in \S\ref{S:prf-qhd}.)  Consulting \eqref{E:wt-rest}, \S\ref{S:prf0}\ref{a:R3} and \S\ref{S:lghd}, we see that $\mu$ must be one of $2\w_1 , \w_2$; in each case $U_\bC$ is real.  The $\ttE$--eigenspace decomposition of the standard representation $U_{\w_1} = \bC^{2r}$ is $\bC^r \op \bC^r$; in particular, the dimensions of the eigenspaces are $(r,r)$.

\smallskip

$\bullet$  In the case that $\mu=2\w_1$, the representation $U_{2\w_1} = \tSym^2\bC^r$ has $\ttE$--eigenspace decomposition $(\tSym^2\bC^r) \op (\bC^r \to \bC^r) \op (\tSym^2\bC^r)$.  The dimensions of the eigenspaces are $(\half r(r+1) \,,\, r^2 \,,\, \half r(r+1) )$.  The requirement $\half r(r+1) = h^{2,0}_\phi \in \{1,2\}$ forces $r = 1$, a contradiction.

\smallskip

$\bullet$  In the case that $\mu = \w_2$, we have $U_{\w_2} \op \tspan_\bC\{Q\} = \tw^2\bC^{2r}$, and the dimensions of the $\ttE$--eigenspaces are $(\half r(r-1) \,,\, r^2-1 \,,\, \half r(r-1) )$.  The requirement $\half r(r-1) = h^{2,0}_\phi \in \{1,2\}$ forces $r = 2$, yielding $\bh_\phi = (1,3,1)$.  This case is covered by Theorem \ref{T:pg=1}\emph{\ref{i:1pd}}.

\subsection{Spinor Hodge domains}

Let $(\fg_\bC,\ttE) = (\fso(2r,\bC) \,,\, \ttA^r)$.  We may assume without loss of generality that $r \ge 4$.  Consulting \S\ref{S:prf0} and \S\ref{S:shd}, we see that $\mu$ is restricted to be one: 
\begin{i_list}
\item \label{i:shd}
$\mu \in \{  \w_1 , 2\w_1 , \w_2 \}$, any $r \ge 4$; 
\item \label{i:shd4}
$r=4$ and $\w \in \{ \w_3 , \w_1+\w_3 , 2\,\w_3 , \w_4 \}$;
\item \label{i:shd5}
$r=5$, $\mu \in \{ \w_4 , \w_5 \}$;
\item \label{i:shd6}
$r=6$, $\mu = \w_5$.
\end{i_list}
We consider each of these cases below.  

\smallskip

\ref{i:shd} If $\mu = \w_1$, then $U_\mu$ is the standard representation $\bC^{2r}$, with $\ttE$--eigenspace decomposition $\bC^r \op \bC^r$.   The dimensions $(r,r)$ of the $\ttE$--eigenspaces are too large ($r\ge 4 > 2 \ge h^{1,1}_\phi$).  If $\mu = 2\w_1$, then $U_{2\w_1} \op \tspan_\bC\{Q\} = \tSym^2\bC^{2r}$, and the dimensions $(\half r (r+1) , r^2-1 , \half r(r+1))$ of the $\ttE$--eigenspaces are again too large.  If $\mu = \w_2$, then $U_{\w_2} = \tw^2\bC^{2r} = (\tw^2\bC^r) \op (\bC^r \ot \bC^r) \op (\tw^2\bC^r)$ and the dimensions $(\half r (r-1) , r^2 , \half r (r-1) )$ are again too large.

\smallskip

\ref{i:shd4} Now suppose that $r=4$.  Then the dimensions of the $\ttE$--eigenspaces for the representations in \ref{i:shd4} are $(8,8)$, $(15,26,15)$, $(10,15,10)$ and 
$(1,6,1)$, respectively.  The requirement that $h^{1,1}_\phi \in \{1,2\}$, restricts us to $\mu = \w_4$.  In this case $U_\mu$ is real, and we have Hodge numbers $\bh_\phi = (1,6,1)$.  This is a special case of Theorem \ref{T:pg=1}\emph{\ref{i:1pd}} under an outer automorphism (triality) of $\fso(8,\bC)$ that permutes the weight $\{ \w_1,\w_3,\w_4\}$. of  

\smallskip

\ref{i:shd5} Next take $r=5$.  The $\ttE$--eigenspaces of $U_{\mu_4}$ and $U_{\w_5}$  have dimensions $(5,10,1)$ and $(1,10,5)$, respectively.  These are too large for our desired Hodge numbers $\bh_\phi$.

\smallskip

\ref{i:shd6} Finally, we consider $r=6$ and $\mu = \w_5$.  In this case $U_{\w_5}$ is quaternionic, and the the $\ttE$--eigenspaces have dimensions $(6,20,6)$ so that the Hodge numbers of the associated Hodge representation $V_\bR = U_{\w_5} \op U_{\w_5}$ are $\bh = (12,40,12)$; again these are too large.

\subsection{Cayley Hodge domains} \label{S:prf-chd}

Let $(\fg_\bC,\ttE) = (\fe_6 \,,\, \ttA^6)$.  Consulting \S\ref{S:prf0} and \S\ref{S:chd}, we see that $\mu$ is restricted to be one of $\{ \w_1 \,,\ \w_2\,,\ \w_6\}$.  In each case $(\mu+\mu^*)(\ttE) = 2$, so that $U_\mu$ has three nontrivial eigenvalues.  The dimensions of the $\ttE$--eigenspaces are $(10,16,1)$, $(16,46,16)$ and $(1,16,10)$, respectively.  In each case the first/last is too large ($>2\ge h^{2,0}_\phi)$ to yield a Hodge representation satisfying the desired constraints.

\subsection{Freudenthal Hodge domains} \label{S:prf-fhd}

Let $(\fg_\bC,\ttE) = (\fe_7 \,,\, \ttA^7)$.  Consulting \S\ref{S:prf0} and \S\ref{S:fhd}, we see that $\mu$ is restricted to be  the first fundamental weight $\mu = \w_1$.  In this case the representation $U_{\w_1}$ is real (with respect to $\fg_\bR$) and we have Hodge numbers $(27,79,27)$; the first is too large ($>2 \ge h^{2,0}_\phi$).

\smallskip

This completes the proofs of Theorems \ref{T:pg=1} and \ref{T:pg=2}.

\subsection{When horizontality fails} \label{S:notH}

In this section we prove Theorem \ref{T:notH}.  Computationally the identification of Hodge subdomains for which horizontality fails entails dropping the assumption that the grading element $\ttE$ is of the form listed in Table \ref{t:herm}.  Fortunately, the period domain $\cD_\bh$ parameterizing weight two, polarized Hodge structures with $p_g=2$ is ``close enough'' to the classical Hermitian period domains (for principally polarized abelian varieties and K3s) that we still have strong restrictions on the possible grading elements.  For this period domain the horizontal subbundle $F^{-1}(T\cD_\bh) \subset T\cD_\bh$ (also known as the \emph{infinitesimal period relation} (IPR)) is a contact subbundle.\footnote{In particular, it has corank one.  In the classical case that the period domain is Hermitian the subbundle $F^{-1}( T \cD_\bh) = T\cD_\bh$ has corank zero.  This is the sense in which the period domain $\cD_\bh$ with $\bh = (2,h^{1,1},2)$ is as ``close as one can get to the classical/Hermitian case.''  (We use the notation $F^{-1}(T\cD_\bh)$ for the horizontal subbundle because it is the first subspace in a natural filtration of the holomorphic tangent bundle $T \check\cD_\bh \supset T \cD_\bh$.)}

Any Hodge structure $\varphi \in \cD_\bh$ induces a Hodge structure on the Lie algebra
\[
  \tilde\fg_\bC \ := \ \tEnd(V_\bC,Q) \ = \ \bigoplus_p \tilde\fg^{p,-p}_\varphi
\]
of $\cG_\bR$ as in \S\ref{S:induced}.  Assuming the normalization of \S\ref{S:norm-ge}, the period domain $\cD_\bh$ is Hermitian if and only if $\tilde\fg^{p,-p}_\varphi = 0$ for all $|p|\ge 2$ (\S\ref{S:horiz}).  And the IPR is contact (as in the present example) if and only if $\tilde\fg^{p,-p}_\varphi = 0$ for all $|p| \ge 3$, and $\tdim\,\tilde\fg^{2,-2}_\varphi = 1$.  Given a Hodge representation \eqref{E:Hrep}, since the induced Hodge structure \eqref{E:ghs} on $\fg_\bR$ is given by $\fg^{p,-p}_\phi = \fg_\bC \,\cap\, \tilde\fg^{p,-p}_\phi$, it follows (from the discussion of \S\ref{S:horiz}) that horizontality will fail for the Hodge subdomain $D \subset \cD_\bh$ if and only if the induced Hodge decomposition \eqref{E:ind} is of the form
\begin{equation}\label{E:contact}
  \fg_\bC \ = \ \fg^{2,-2}_\phi \,\op\, \fg^{1,-1}_\phi \,\op\,
  \fg^{0,0}_\phi \,\op\, \fg^{-1,1}_\phi \,\op\, \fg^{-2,2}_\phi \,, 
\end{equation}
with $\tdim_\bC\,\fg^{2,-2}_\phi = 1$.  In this case the grading element is necessarily of the form listed in Table \ref{t:contact}, cf.~\cite[Proposition 3.2.4]{MR2532439}.\footnote{Be aware there are typos in the table of that proposition.}  (See \cite{MR1920389} for remaining notation.)
\begin{table}[b]
\caption{Data underlying irreducible contact Hodge domains}
\[
\begin{array}{|ccccc|}
  \hline
  \fg_\bC & \ttE & \check D = G_\bC/P_\ttE & \fg_\bR & \fk \\ \hline
  \fsl(r+1,\bC) & \ttA^1+\ttA^r & \tFlag(1,r;\bC^{r+1}) & \fsu(2,r-1) 
  & \mathfrak{s}(\fu(2) \op \fu(r-1)) \\
  \fso(d+4,\bC) & \ttA^2 & \tGr^Q(2,\bC^{d+4}) & \fso(4,d) 
  & \mathfrak{s}(\fo(4) \op \fo(d)) \\
  \fsp(2r,\bC) & \ttA^1 & \bP^{2r-1} & \fsp(1,r-1) 
  & \fsp(1) \op \fsp(r-1) \\
  \fe_6 & \ttA^2 & & \mathrm{E II} 
  & \fsu(6) \op \fsu(2) \\
  \fe_7 & \ttA^1 & & \mathrm{E VI} 
  & \fso(12) \op \fsu(2) \\
  \fe_8 & \ttA^8 & & \mathrm{E IX} 
  & \fe_7 \op \fsu(2)  \\
  \ff_4 & \ttA^1 & & \mathrm{F I}
  & \fsp(3) \op \fsu(2) \\
  \fg_2 & \ttA^2 & & \mathrm{G}
  & \fsu(2) \op \fsu(2) \,.
  \\ \hline
\end{array}
\]
\label{t:contact}
\end{table}
For each of the five exceptional cases, the compact dual $\check D = \cG/\cP_\mathtt{i} \inj \bP V_{\w_\tti}$ is a rational homogeneous variety with isotropy group $\cP_\tti$ the maximal parabolic subgroup associated with the grading element $\ttE = \ttA^\tti$.

The proof of Theorem \ref{T:notH} now proceeds as outlined in \S\ref{S:prf0}, the single exception being that we work with Table \ref{t:contact} (not Table \ref{t:herm}).  As it is a straightforward variation on the proof of Theorems \ref{T:pg=1} and \ref{T:pg=2}, the proof is left to the reader; for easy reference, the relevant eigenvalues are listed in \S\ref{S:ev-contact}.

\subsection{When simplicity fails} \label{S:ss}

Here we prove Theorem \ref{T:ss}.  To begin, suppose that $\fg_\bC = \fg_1 \op \fg_2$ factors into the direct sum of two nontrivial ideals.  Then $U = U_\mu$ is necessarily of the form $T_1 \ot T_2$ with $T_i$ an irreducible representation of $\fg_i$ of highest weight $\mu_i$ and $\mu = \mu_1+\mu_2$.  Likewise, $\ttE = \ttE_1 + \ttE_2$, with $\ttE_i$ a grading element of $\fg_i$.  We write
\[
  (\fg_\bC , \ttE , \mu) \ = \ 
  (\fg_1 , \ttE_1 , \mu_1) \,\op\, (\fg_2,\ttE_2,\mu_2) \,.
\]
The Hodge representation will have weight/level $n=2$ if and only if 
\[
  1 \ = \ c \,+\, \mu(\ttE) \ = \ c \,+\, \mu_1(\ttE_1) \,+\, \mu_2(\ttE_2) \,.
\]
Recall Remark \ref{R:ev}, and note that $e(\mu,\ttE) = e(\mu_1,\ttE_1)+e(\mu_2,\ttE_2) \ge 2$.  The hypothesis $e(\mu,\ttE)=2$ forces $e(\mu_i,\ttE_i) = 1$ and $\fg_i$ to be simple.

\begin{proposition} \label{P:n=2ss}
Any semisimple algebra $\fg_\bR$ admitting a Hodge representation of level $n=2$ is either simple, or decomposes as the sum $\fg_\bR = \fg_{1,\bR} \op \fg_{2,\bR}$.  In the latter case, the triples $(\fg_i,\ttE_i,\mu_i)$ are necessarily one of:
\begin{i_list_emph}
\item \label{i:ss1}
$(\fsl_{r+1}\bC , \ttA^a , \w_1)$.  The (standard) representation $U_{\w_1} = \bC^{r+1}$ is real if $r=1$, and complex otherwise.  The $\ttA^a$--eigenspace decomposition is $\bC^{r+1} = \bC^a \op \bC^{r+1-a}$.
\item \label{i:ss2}
$(\fsl_{r+1}\bC , \ttA^1 , \w_a)$.  The representation $U_{\w_a} = \tw^a \bC^{r+1}$ is complex unless $r+1 = 2a$, in which case the representation is real if and only if $a$ is odd.  The $\ttA^1$--eigenspace decomposition $\tw^a \bC^{r+1} = (\bC^1 \ot \tw^{a-1} \bC^r) \op (\tw^a \bC^r)$ is induced by that of $\bC^{r+1} = \bC \op \bC^r$.
\item \label{i:ss3}
$(\fsp(2r,\bC) , \ttA^r , \w_1)$.  The (standard) representation $U_{\w_1} = \bC^{2r}$ is real, and has $\ttA^r$--eigenspace decomposition $\bC^{2r} = \bC^r \op \bC^r$.
\item
$(\fso(2r,\bC) , \ttA_r , \w_1)$.  The (standard) representation $U_{\w_1} = \bC^{2r}$ is quaternionic, and has $\ttA^r$--eigenspace decomposition $\bC^{2r} = \bC^r \op \bC^r$.
\item
$(\fso(2r+1,\bC) , \ttA^1 , \w_r)$.  The (spin) representation $U_{\w_r}$ is real if $\half r(r-1)$ is even, and quaternionic otherwise.  The $\ttA^1$--eigenspace decomposition is $U_{\w_r} = \bC^{2^{r-1}} \op \bC^{2^{r-1}}$.
\item
$(\fso(2r,\bC) , \ttA^1 , \w_r)$.  The (half-spin) representation $U_{\w_r}$ is complex if $r$ is odd.  If $r$ is even, then the representation is real if $\half (r+1) (r-2)$ is even, and quaternionic otherwise.  The $\ttA^1$--eigenspace decomposition is $U_{\w_r} = \bC^{2^{r-2}} \op \bC^{2^{r-2}}$.
\end{i_list_emph}
\end{proposition}

\begin{proof}
The proof proceeds as outlined in \S\ref{S:prf0} and demonstrated in \S\S\ref{S:prf-ghd}--\ref{S:prf-fhd}; details are left to the reader.
\end{proof}

Let 
\[
  T_i \ = \ T_{i,\mu_i(\ttE_i)} \ \op \ T_{i,\mu_i(\ttE_i)-1}
  \ = \ T_i' \ \op \ T_i''
\]
be the $\ttE_i$--eigenspace decompositions.  Then the $\ttE$--eigenspace decomposition 
\[
  U \ = \ U_{\mu(\ttE)} \ \op \ U_{\mu(\ttE)-1} \ \op \ U_{\mu(\ttE)-2} 
\]
is given by 
\begin{eqnarray*}
  U_{\mu(\ttE)} & = & T_1' \,\ot\, T_2' \\
  U_{\mu(\ttE)-1} & = & 
  ( T_1' \,\ot\, T_2'' ) \ \op \ 
  (T_1'' \,\op\, T_2' )\\
  U_{\mu(\ttE)-2} & = & 
  T_1'' \,\ot\, T_2'' \,.
\end{eqnarray*}
So in order to obtain a Hodge representation \eqref{E:Hrep} with $p_g = h^{2,0} = 2$ we must have 
\[
 1 \ \le \ \tdim\,T_1' \,\ot\, T_2' \,,\quad 
  \tdim\,T_1'' \,\ot\, T_2'' \ \le \ 2 \,; 
\]
in particular, 
\[
  \tdim\,T_i \ \le \ 4 \,.
\]
Modulo isomorphisms of (low-rank) Lie algebras, this leaves us with \emph{\ref{i:ss1}} and \emph{\ref{i:ss3}} of Proposition \ref{P:n=2ss}.  Theorem \ref{T:ss} now follows from the discussion of \S\ref{S:prf0}; details are left to the reader.

\section{Hodge representations of Calabi--Yau type} \label{S:CY}

We say that a Hodge representation \eqref{E:Hrep} is of \emph{Calabi--Yau type} (or \emph{CY-type}) if the first Hodge number $h^{n,0}_\phi = 1$.  The irreducible CY-Hodge representations with $\fg_\bR$ semisimple are classified in \cite[Proposition 6.1]{MR3217458}.  They are precisely the tuples $(\fg_\bC,\ttE,\mu,c)$ of Theorem \ref{T:strategy} with $c=0$ (\S\ref{S:rvss}), and such that:
\begin{a_list}
\item \label{a:cy}
$\mu^i = 0$ whenever $\a_i(\ttE) = 0$, where $\a_i$ are the simple roots of (the semisimple) $\fg_\bC$ and the $0 \le \mu^i \in \bZ$ are the coefficients of $\mu = \mu^i \w_i$ as a linear combination of the fundamental weights $\w_i$; 
\item 
either the representation is real (equivalently, $U = U^*$ and $\mu(\ttT_\phi)$ is an even integer), or
\item 
$\mu(\ttE_\phi) \not= \mu^*(\ttE_\phi)$, and $U$ is necessarily complex.
\end{a_list}

\begin{remark}\label{R:a}
The condition \ref{a:cy} above is equivalent to the statement that $\tdim\,U_{\mu(\ttE)} = 1$; equivalently, $U_{\mu(\ttE)}$ is a highest weight line.
\end{remark}

\begin{theorem}\label{T:CY}
An irreducible Hodge representation \eqref{E:Hrep} is of CY-type if and only if the corresponding tuple $(\fg_\bC,\ttE,\mu,c)$ of Theorem \ref{T:strategy} has the properties:
\begin{i_list_emph}
\item The condition \eref{a:cy} above holds.
\item If $U_\mu$ is not real (with respect to the semisimple $\fg_\bR$), then $\mu(\ttE)+c > \mu^*(\ttE)-c$.
\end{i_list_emph}
\end{theorem}

\begin{proof}
It is straightforward to deduce the theorem from the proof of \cite[Proposition 6.1]{MR3217458} and the discussion of \S\ref{S:rvss}.  Details are left to the reader.
\end{proof}

\begin{example} \label{eg:K3}
The (rational) Hodge groups of K3 type (CY 2-fold type) were determined by Zarhin \cite{MR697317}.  The corresponding (real) Hodge representations \eqref{E:Hrep} are those with Hodge numbers $\bh_\phi = (1,h,1)$.  The list of all associated tuples (Theorem \ref{T:strategy}) is 
\begin{i_list}
\item
$(\fso_{h+2}\bC , \ttA^1 , \w_1 , 0)$, with $\bh_\phi = (1,h,1)$ and $h \ge 3$.\footnote{The associated domain $D_\phi$ is the period $\cD$ parameterizing polarized Hodge structures of K3-type with $\bh_\phi = (1,h,1)$, cf.~Example \ref{eg:pd}.}
\item 
$(\fsl_2\bC\op\fsl_2\bC , \ttA^1+\ttA^2 , \w_1+\w_2 , 0)$, with $\bh_\phi = (1,2,1)$.\footnote{Recall that $\fso_4\bC = \fsl_2\bC \op \fsl_2\bC$ is semisimple.  Here $D_\phi$ is again the period domain.}
\item
$(\fsl_{r+1}\bC , \ttA^1, \w_1, \frac{1}{r+1})$ with $\bh_\phi = (1,2r,1)$ and $r \ge 2$.
\item
$(\fsl_2\bC , \ttA^1 , 2\w_1 , 0 )$ with $\bh_\phi = (1,1,1)$.
\item
$(\fsl_4\bC,\ttA^2,\w_2 , 0)$, with $\bh_\phi = (1,4,1)$.
\end{i_list}
\end{example}

\begin{example} \label{eg:CY3}
The set of all Hodge representations \eqref{E:Hrep} of CY 3-fold type (Hodge numbers $\bh_\phi = (1,h,h,1)$) is enumerated in \cite{HHan}.  In the case that $D_\phi$ is \emph{horizontal} (and therefore Hermitian) these are of particular interest \cite{MR3189469, MR1258484}.  The corresponding (horizontal) tuples $(\fg_\bC,\ttE,\mu,c)$ of Theorem \ref{T:strategy} are 
\begin{i_list}
\item
$(\fsl_2\bC , \ttA^1 , 3 \w_1 , 0 )$ with $\bh_\phi = (1,1,1,1)$.
\item
$(\fsl_2\bC , \ttA^1 , \w_1 , 0 )^{\op3}$ with $\bh_\phi = (1,3,3,1)$.
\item
$(\fsl_6\bC , \ttA^3 , \w_3 , 0 )$ with $\bh_\phi = (1,9,9,1)$.
\item
$(\fsl_{r+1}\bC , \ttA^1 , \w_1 , \frac{3}{2}-\frac{r}{r+1})$ with $\bh_\phi = (1,r,r,1)$.
\item
$(\fsl_{r+1}\bC , \ttA^1 , 2\w_1 , \frac{3}{2}-\frac{2r}{r+1})$ with $\bh_\phi = (1,h,h,1)$ and $h+1 = \half(r+1)(r+2)$.
\item
$(\fsl_{r+1}\bC , \ttA^2 , \w_2 , \frac{3}{2}-\frac{2(r-1)}{r+1})$ with $\bh_\phi = (1,h,h,1)$ and $h+1 = \half r(r+1)$.
\item
$(\fsl_{r+1}\bC,\ttA^1,\w_1) \op (\fsl_{r'+1}\bC , \ttA^1 , \w_1 )$ and $c=\frac{3}{2} - \frac{r}{r+1} - \frac{r'}{r'+1}$, with $\bh_\phi = (1,h,h,1)$ and $h = r+r'+rr'$.
\item
$(\fsl_2\bC , \ttA^1 , \w_1 , 0 ) \op (\fg'_\bC,\ttE',\mu',c')$ with $\bh_\phi = (1,h'+1,h'+1,1)$, where $(\fg'_\bC,\ttE',\mu',c')$ is any tuple of Example \ref{eg:K3} with Hodge numbers $\bh'=(1,h',1)$.
\item
$(\fsp_6\bC , \ttA^3 , \w_3 , 0 )$ with $\bh_\phi = (1,6,6,1)$.
\item
$(\fso_m\bC , \ttA^1 , \w_1 , 1/2 )$ with $\bh_\phi = (1,m-1,m-1,1)$.
\item
$(\fso_{10}\bC , \ttA^5 , \w_5 , 1/4 )$ with $\bh_\phi = (1,15,15,1)$.
\item
$(\fso_{12}\bC , \ttA^6 , \w_6 , 0)$ with $\bh_\phi = (1,15,15,1)$.
\item
$(\fe_6 , \ttA^6 , \w_6 , 1/6)$ with $\bh_\phi = (1,26,26,1)$.
\item
$(\fe_7 , \ttA^7 , \w_7 , 0)$ with $\bh_\phi = (1,27,27,1)$.
\end{i_list}
\end{example}

\appendix

\section{Duality and eigenvalues for Hodge representations}

For the computations of \S\S\ref{S:eg}--\ref{S:notH} it is useful to make some general observations about the irreducible $\fg_\bC$--representations $U_\mu$ that yield Hodge representations $V_\bR$ for each such pair.  (Those that follow are all elementary consequences of the representation theory of complex, simple Lie algebras and may be found in any standard text.)  Throughout we let $\{ \w_1,\ldots,\w_r\} \subset \fh^*$ denote the fundamental weights of $\fg_\bC$, and write the dominant integral weight $\mu = \mu^i \w_i$ with $0 \le \mu^i \in \bZ$.  

Motivated by the considerations of \S\ref{S:rcq}, we define
\[
  \ttT_i \ := \ 2 \sum_{j\not=i} \ttA^j \,.
\]

\subsection{Duality}

Every representation $U_\mu$ of $\fg_\bC = \fso(2r+1,\bC)$, $\fsp(2r,\bC)$, $\fe_7$, $\fe_8$, $\ff_4$ and $\fg_2$ is self-dual.
\begin{numlist}
\item
A $\fg_\bC = \fsl(r+1,\bC)$ representation $U_\mu$ is self-dual if and only if $\mu^i = \mu^j$ for all $i+j = r+1$.  
\item
A $\fg_\bC = \fso(2r,\bC)$ representation $U_\mu$ fails to be self-dual if and only if $r$ is odd and $\mu^{r-1}\not=\mu^r$.
\item
For $\fg_\bC = \fe_6$, we have $\w_1^* = \w_6$, $\w_2^* = \w_2$, $\w_3^* = \w_5$, $\w_4^* = \w_4$.
\end{numlist}

\subsection{Hermitian symmetric domains} \label{S:ev-herm}

Let $(\fg_\bC,\ttE)$ be the data underlying the irreducible Hermitian symmetric domains (Table \ref{t:herm}).

\subsubsection{Grassmannian Hodge domains} \label{S:ghd}

With the notation in the first row of Table \ref{t:herm}, we have $r+1 = a+b$.  Given Remark \ref{R:ev} it will be useful to note that, given $i+j = r+1$, we have
\begin{eqnarray*}
  (\w_i + \w_j)(\ttA^a) & = &
  \left\{
  \begin{array}{ll}
    a \,,\quad & a \le i \le j  \,,\\
    i \,,\quad & i \le a \le j \,,\\
    b \,,\quad & i \le j \le a \,;
  \end{array}
  \right. \\
  (\w_i+\w_j)(\ttT_a) & = & 2 ij \,-\,(\w_i + \w_j)(2\,\ttA^a) \,.
\end{eqnarray*}

\subsubsection{Quadric Hodge domains} \label{S:qhd}

Here (the second row of Table \ref{t:herm}) the rank $r$ of $\fg_\bC$ is given by $d+2 \in \{ 2r,2r+1\}$.  

In the case that $d\equiv 1$ mod $2$, we have 
\begin{eqnarray*}
  \w_i(\ttA^1) & = &  
  \left\{
  \begin{array}{ll}
    1 \,,\quad & i \le r-1 \,,\\
    \half \,,\quad & i=r \,.
  \end{array}
  \right.\\
  \w_i(\ttT_1) & \equiv & 0 \quad\hbox{mod } \ 2 \,,\quad i \le r-1 \,,\\
  \w_r(\ttT_1) & = & \half (r-1) (r+2) \,.
\end{eqnarray*}
In the case that $d\equiv 0$ mod $2$, we have 
\begin{eqnarray*}
  \w_i(\ttA^1) & = &  
  \left\{
  \begin{array}{ll}
    1 \,,\quad & i \le r-2 \,,\\
    \half \,,\quad & i=r-1,r \,.
  \end{array}
  \right.\\
  \w_i(\ttT_1) & \equiv & 0 \quad\hbox{mod } \ 2 \,,\quad i\le r-2 \,,\\
  \w_{r-1}(\ttT_1) \ = \ \w_r(\ttT_1) & = & \half (r-2)(r+1) \,.
\end{eqnarray*}

\subsubsection{Lagrangian Grassmannian Hodge domains} \label{S:lghd}

We have
\begin{eqnarray*}
  \w_i(\ttA^r) & = & \half i \,,\\
  \w_i(\ttT_r) & \equiv & 0 \quad\hbox{mod } \ 2 \,.
\end{eqnarray*}
In particular, every representation $U_\mu$ is real (\S\ref{S:rcq}), with respect to the data $(\fg_\bC,\ttE) = (\fsp(2r,\bC) \,,\, \ttA^r)$. 

\subsubsection{Spinor Hodge domains} \label{S:shd}

We have 
\[
\begin{array}{rclrcl}
  \w_i(\ttA^r) & = & \half i \,,\quad & 
  \w_i(\ttT_r) & \equiv & i \quad \hbox{mod } \ 2 \,,\quad i \le r-2 \,;\\
  \w_{r-1}(\ttA^r) & = & \fourth (r-2) \,, \quad & 
  \w_{r-1}(\ttT_r) & = & \half(r^2-2r+2) \,;\\
  \w_r(\ttA^r) & = & \fourth r \,, \quad &
  \w_r(\ttT_r) & = & \half r (r-2) \,.
\end{array}
\]

\subsubsection{Cayley Hodge domains} \label{S:chd}

We have   
\[
\begin{array}{rclrcl}
  (\w_1 + \w_6)(\ttA^6) & = & 2 \,,\quad &
  (\w_1 + \w_6)(\ttT_6) & \equiv & 0 \quad\hbox{mod } \ 2 \,,\\
  \w_2(\ttA^6) & = & 1 \,,\quad &
  \w_2(\ttT_6) & \equiv & 0 \quad\hbox{mod } \ 2 \,,\\
  (\w_3 + \w_5)(\ttA^6) & = & 3 \,,\quad &
  (\w_3 + \w_5)(\ttT_6) & \equiv & 0 \quad\hbox{mod } \ 2 \,,\\
  \w_4(\ttA^6) & = & 2 \,,\quad &
  \w_4(\ttT_6) & \equiv & 0 \quad\hbox{mod } \ 2 \,.
\end{array}
\]
In particular, every representation is either real or complex with respect to the data $(\fg_\bC , \ttE) = (\fe_6,\ttA^6)$.

\subsubsection{Freudenthal Hodge domains} \label{S:fhd}

Every representation is self-dual.  We have 
\[
\begin{array}{rclrclrclrcl}
  \w_1(\ttA^7) & = & 1 \,,\quad &
  \w_2(\ttA^7) & = & 3/2 \,,\quad &
  \w_3(\ttA^7) & = & 2 \,,\quad &
  \w_4(\ttA^7) & = & 3 \,,\\
  \w_5(\ttA^7) & = & 5/2 \,,\quad &
  \w_6(\ttA^7) & = & 2 \,,\quad &
  \w_7(\ttA^7) & = & 3/2 \,, 
\end{array}
\]
and $\w_i(\ttT_7) \equiv 0$ mod $2$, for all $1 \le i \le 7$.  So every representation is real (\S\ref{S:rcq}) with respect to the data $(\fg_\bC,\ttE) = (\fe_7,\ttA^7)$.

\subsection{Contact domains} \label{S:ev-contact}

Let $(\fg_\bC,\ttE)$ be the data underlying the irreducible contact Hodge domains (Table \ref{t:contact}).  (Duality of representations is as in \S\ref{S:ev-herm} and so will not be repeated here.)

\subsubsection{Special Linear}

We begin with the first row of Table \ref{t:contact}.  We have $\w_i(\ttA^1 + \ttA^r) = 1$ for all $1 \le i \le r$.  We have $\ttT_\phi = 2 (\ttA^2 + \cdots + \ttA^{r-1})$.  If $r=3$, then  $\w_1(\ttT_\phi) = 1 = \w_3(\ttT_\phi)$ and $\w_2(\ttT_\phi) = 2$.  If $r \ge 4$, then $\w_i(\ttT_\phi) \equiv 0$ mod $2$.

\subsubsection{Orthogonal}

Here (the second row of Table \ref{t:contact}) the rank $r$ of $\fg_\bC$ is given by $d+4 \in \{ 2r,2r+1\}$.

In the case that $d\equiv 1$ mod $2$, we have 
\begin{eqnarray*}
  \w_i(\ttA^2) & = &  
  \left\{
  \begin{array}{ll}
    1 \,,\quad & i=1,r \,,\\
    2 \,,\quad & 2\le i \le r-1 \,.
  \end{array}
  \right.\\
  \w_i(\ttT_2) & \equiv & 0 \quad\hbox{mod } \ 2 \,,\quad i \le r-1 \,,\\
  \w_r(\ttT_2) & \equiv & \half r (r+1) \quad\hbox{mod } \ 2 \,.
\end{eqnarray*}
In the case that $d\equiv 0$ mod $2$, we have 
\begin{eqnarray*}
  \w_i(\ttA^2) & = &  
  \left\{
  \begin{array}{ll}
    1 \,,\quad & i =1,r-1,r \,,\\
    2 \,,\quad & 2 \le i \le r-2 \,.
  \end{array}
  \right.\\
  \w_i(\ttT_2) & \equiv & 0 \quad\hbox{mod } \ 2 \,,\quad i\le r-2 \,,\\
  \w_{r-1}(\ttT_2) \ = \ \w_r(\ttT_2) & \equiv 
  & \half r (r+1) \quad\hbox{mod } \ 2 \,.
\end{eqnarray*}

\subsubsection{Symplectic}

We have
\begin{eqnarray*}
  \w_i(\ttA^1) & = & 1 \,,\\
  \w_i(\ttT_1) & \equiv & i \quad\hbox{mod } \ 2 \,.
\end{eqnarray*}

\subsubsection{Exceptional, rank 6}

We have $(\fg_\bC,\ttE) = (\fe_6 , \ttA^2)$ and   
\[
\begin{array}{rclrcl}
  (\w_1 + \w_6)(\ttA^2) & = & 2 \,,\quad &
  (\w_1 + \w_6)(\ttT_2) & \equiv & 0 \quad\hbox{mod } \ 2 \,,\\
  \w_2(\ttA^2) & = & 2 \,,\quad &
  \w_2(\ttT_2) & \equiv & 0 \quad\hbox{mod } \ 2 \,,\\
  (\w_3 + \w_5)(\ttA^2) & = & 4 \,,\quad &
  (\w_3 + \w_5)(\ttT_2) & \equiv & 0 \quad\hbox{mod } \ 2 \,,\\
  \w_4(\ttA^2) & = & 3 \,,\quad &
  \w_4(\ttT_2) & \equiv & 0 \quad\hbox{mod } \ 2 \,.
\end{array}
\]


\subsubsection{Exceptional, rank 7}

We have $(\fg_\bC,\ttE) = (\fe_7 , \ttA^1)$, with $\w_7(\ttA^1) = 1$ and $\w_i(\ttA^1) \ge 2$ for all $1 \le i \le 6$.  Also $\w_1(\ttT_1) \equiv 0$ mod $2$.


\subsubsection{Exceptional, rank 8}

For $(\fg_\bC,\ttE) = (\fe_8 , \ttA^8)$ we have $\w_i(\ttA^8) \ge 2$ for all $1 \le i \le 8$.


\subsubsection{Exceptional, rank 4}

We have $(\fg_\bC,\ttE) = (\ff_4 , \ttA^1)$ and $\w_i(\ttA^1) \ge 2$ for all $1 \le i \le 4$.


\subsubsection{Exceptional, rank 2}

We have $(\fg_\bC,\ttE) = (\fg_2 , \ttA^2)$, $\w_1(\ttA^2) = 1$ and $\w_2(\ttA^2) = 2$, and $\w_i(\ttT_1) \equiv 0$ mod $2$ for all $1 \le i \le 2$.

\bibliography{HodgeReps.bbl}
\bibliographystyle{alpha}

\end{document}